\documentclass[12pt]{amsart}
\usepackage[english]{babel}
\usepackage{amsrefs,amssymb,amsmath,amsthm,amsfonts,amscd}
\usepackage{epsfig,graphicx,psfrag,color}
\usepackage[utf8]{inputenc}
\usepackage{xcolor}
\usepackage{floatrow}
\usepackage{wrapfig}
\usepackage{subfigure}
\usepackage{caption}
\usepackage{float}
\usepackage{enumitem}
\usepackage{hyperref}
\hypersetup{
    colorlinks,
    citecolor=black,
    filecolor=blue,
    linkcolor=black,
    urlcolor=blue
}
\usepackage[a4paper,
            bindingoffset=0in,
            left=1.5cm,
            right=1.5cm,
            top=1.5cm,
            bottom=1.5cm,
            footskip=.25in]{geometry}
\graphicspath{{figuras/}}

\DeclareMathOperator{\per}{Per}

\DeclareMathOperator{\ho}{\mathcal{H}}
\DeclareMathOperator{\homeo}{\ho}

\DeclareMathOperator{\shift}{\sigma}
\DeclareMathOperator{\tshift}{\tilde\shift}

\theoremstyle{plain}
\newtheorem{thm}{Theorem}[section]
\newtheorem{cor}[thm]{Corollary}
\newtheorem{lem}[thm]{Lemma}
\newtheorem{prop}[thm]{Proposition}

\newtheorem{que}[thm]{Question}

\theoremstyle{definition}
\newtheorem{rmk}[thm]{Remark}
\newtheorem{ex}[thm]{Example}
\newtheorem{exa}[thm]{Example}
\newtheorem{df}[thm]{Definition}

\theoremstyle{remark}

\DeclareMathOperator{\diam}{diam}

\DeclareMathOperator{\clos}{clos}

\DeclareMathOperator{\dist}{d}
\DeclareMathOperator{\distcc}{\dist_{C^0}}
\DeclareMathOperator{\tdist}{\tilde{d}}
\DeclareMathOperator{\ttdist}{\tilde{\tilde{d}}}

\DeclareMathOperator{\orb}{orb}

\DeclareMathOperator{\con}{con}
\DeclareMathOperator{\shm}{Sh}
\DeclareMathOperator{\pom}{Po}

\newcommand{\R}{\mathbb R}

\newcommand{\Z}{\mathbb Z}

\renewcommand{\epsilon}{\varepsilon}

\newcommand{\expc}{\xi}
\newcommand{\tM}{\tilde{M}}

\newcommand{\nUparrow}{\Uparrow\hspace{-.34cm}\diagup\hspace{.31cm}}

\begin{document}

\title{Shadowing maps}

\author{Alfonso Artigue}
\address{Departamento de Matemática y Estadística del Litoral\\
Centro Universitario Regional Litoral Norte\\
Universidad de la Rep\'ublica\\
Florida 1065, Paysandú, Uruguay.\\}
\email{aartigue@litoralnorte.udelar.edu.uy}

\date{\today}
\begin{abstract}
	This article is about the shadowing property of homeomorphisms on compact metric spaces and the map associating a point of the space to each pseudo-orbit, called 'shadowing map'. Based on some particular dynamical properties, as expansivity, we develop a brief theory and a hierarchy
	of such maps. We consider examples as odometers, shifts on infinite spaces, topologically hyperbolic homeomorphisms and north-shouth dynamics.
	We revisit a well-known technique for proving shadowing of expansive homeomorphisms with canonical coordinates due to R. Bowen, to obtain a  shadowing map with the property we call 'self-tuning' from a hyperbolic bracket. This notion is introduced as part of the hierarchy of shadowing maps studied in this paper.
\end{abstract}
\maketitle


\section{Introduction}
\label{secIntro}
Let $f\colon M\to M$ be a homeomorphism of the compact metric space $(M,\dist)$.
For $\delta>0$ given we say that a sequence $x_i\in M$, $i\in\Z$, is a $\delta$-\textit{pseudo orbit} (or simply a $\delta$-\textit{orbit}) if
$\dist(x_{i+1},f(x_i))\leq\delta$ for all $i\in\Z$.
The homeomorphism is said to have the \textit{shadowing property} (also known as \textit{pseudo-orbit tracing property}) if for all $\epsilon>0$ there is $\delta>0$ such that for each $\delta$-orbit $x$ there is $p\in M$ such that
$\dist(f^i(p),x_i)\leq\epsilon$ for all $i\in\Z$.
We say that such $p$ is an $\epsilon$-\textit{shadow} of the $\delta$-orbit $x$.

The definition of the shadowing property is about a function associating to each $\delta$-orbit a point of the space.
In this article we study some properties of such function, their relations and the dynamical consequencies.
We assume two basic axioms:
\begin{itemize}
	\item continuity and
	\item orbits are shadowed by themselves.
\end{itemize}
For the continuity we consider the $\delta$-orbits as elements of $M^\Z$ endowed with the product topology.
For the second axiom notice that for $\delta=0$ we have a \textit{real orbit}, \textit{i.e.}, $x_i=f^i(x_0)$ for all $i\in\Z$. For such an orbit $x$ we demand the \textit{shadowing point} to be $x_0$. These are quite natural conditions but, alone, they do not imply the shadowing property, see Example \ref{ExProjPOM}. A map satisfying these two axioms will be called \textit{pseudo-orbit map}, see \S\ref{secPOM}.
A pseudo-orbit map inducing the shadowing property will be called \textit{shadowing map}, see \S\ref{secShm}.
In this article we develope a hierarchy of shadowing maps by adding some properties to these elementary axioms.

There are some families of examples with a canonical shadowing map which guided our research.
The strongest family in our set of examples is that of
expansive homeomorphisms with shadowing.
We say that $f$ is \textit{expansive} if there is $\expc>0$ such that if $p\neq q$, $p,q\in M$, then there is $i\in\Z$ such that $\dist(f^i(p),f^i(q))>\expc$.
Expansive homeomorphisms with shadowing are called \textit{topologically hyperbolic}.
This notion is at the top of the hierarchy shown in Table \ref{tablaShm}.

\begin{table}[ht]
	\[
	\begin{array}{c}
		\hbox{\fbox{Topological hyperbolicity}}
		\\
		\Updownarrow \text{Theorem \ref{thmTopHypChar}}
		\\
		\hbox{\fbox{Dynamically-invariant shift-invariant shadowing map}}
		\\
		\Downarrow \text{Obvious}
		\\
		\hbox{\fbox{Shift-invariant shadowing map}}
		\\
		\Downarrow \text{Remark \ref{rmkshInvImpAShInv}}
		\\
		\hbox{\fbox{Self-tuning shadowing map}}
		\\
		\Downarrow \text{Proposition \ref{propAlmSHMimpL-shm}}
		\\
		\hbox{\fbox{L-shadowing map}}
		\\
		\Downarrow\text{Proposition \ref{propLshmImplicaShm}}
		\\
		\hbox{\fbox{Shadowing map}}
		\\
		\Downarrow \text{Proposition \ref{propShmSH}}
		\\
		\hbox{\fbox{Shadowing}}
	\end{array}
	\]
	\caption{Hierarchy of shadowing maps.}
	\label{tablaShm}
\end{table}
From the definitions we already see an important property: if $\epsilon\leq\expc/2$ and $\delta$ is taken from the definition of shadowing then for a $\delta$-orbit there is a unique shadowing point.
The uniqueness implies two properties:
\begin{itemize}
	\item dynamical-invariance and
	\item shift-invariance.
\end{itemize}
To fix notation denote as $\shm_f(x)$ the point that shadows the pseudo-orbit $x$.
On the one hand, we can transform $x$ by applying $f$ to obtain the pseudo-orbit $y$ with $y_i=f(x_i)$ for all $i\in\Z$, which is shadowed by $\shm_f(y)$. But notice that $f(\shm_f(x))$ also shadows $y$. The uniqueness of the shadowing point implies the dynamical-invariance: $\shm_f(f(x))=f(\shm_f(x))$.
On the other hand, $x$ can be transformed by a shift to obtain $z=\shift(x)$ where $z_i=x_{i+1}$ for all $i\in\Z$.
Again, $f(\shm_f(x))$ is a shadow of $\shift(x)$ and the uniqueness implies the shift-invariance:
$\shm_f(\shift(x))=f(\shm_f(x))$.

If we forget expansivity and dynamical-invariance we have shift-invariant shadowing maps. In \S\ref{secShInvShm} we prove several properties for systems with such maps. Also, we show that shift homeomorphisms of compact spaces with infinitely many points have a shift-invariant shadowing map which is not dynamically-invariant.

In \cite{Bowen75} R. Bowen proves the shadowing property for expansive homeomorphisms with canonical coordinates.
This is done by considering the hyperbolic contraction of local stable and unstable sets.
In \cite{Ruelle} D. Ruelle defines Smale spaces in terms of a bracket with  hyperbolic behaviour.
The canonical coordinates induces a bracket as $[p,q]=r$ where $r$ is the point in the local stable set of $q$ and in the local unstable set of $p$, provided that $p$ and $q$ are close enough.
In \S\ref{secAlmShInvShm} we revisit Bowen's technique for hyperbolic brackets. We forget the expansivity condition and look for minimal conditions on the bracket in order to induce a shadowing map. In \S\ref{secAlmShInvShm} we prove that Bowen's technique gives more than a shadowing map,
see Theorem \ref{thmalmshinvshmbracket}.
This extra technical condition will be called \textit{self-tuning}.
It has the following dynamical meaning: if $p\in M$ shadows the pseudo-orbit $x$ and for some $j\in\Z$ and $n\geq 1$ large the distances $\dist(f(x_i),x_{i+1})$ are small for all $|i-j|\leq n$ then the shadowing is more precise, \textit{i.e.} $\dist(f^j(p),x_j)$ is small.
In other words, the shadowing is better where the pseudo-orbit has smaller jumps.
In \S\ref{secNS} we show that Bowen's technique for a north-south homeomorphism gives rise to a self-tuning shadowing map which is not shift-invariant.
Between shadowing maps and self-tuning shadowing maps we place the \textit{L-shadowing maps}; which are shadowing maps inducing the L-shadowing property \cites{CaCo19,GMMT}
(also called \textit{two-sided s-limit shadowing}).
In Table \ref{tablaShm2} we list the counterexamples we found.
\begin{table}[ht]
	\[
	\begin{array}{c}
		\hbox{\fbox{Topological hyperbolicity}}\\
		\nUparrow \text{Remark \ref{rmkShiftShiftInvShmNoExp} (shift maps)}\\
		\hbox{\fbox{Shift-invariant shadowing map}}\\
		\nUparrow\text{Proposition \ref{propNoShInvNS}  (north-south)}\\
		\hbox{\fbox{Self-tuning shadowing map}}\\
		\Uparrow?\\
		\hbox{\fbox{L-shadowing map}}\\
		\nUparrow\text{Proposition \ref{propOdoContraEj} (h-shadowing)}\\
		\hbox{\fbox{Shadowing map}}\\
		\nUparrow\text{Example \ref{exaPANoShm} (pseudo-Anosov of the two-sphere)}\\
		\hbox{\fbox{Shadowing}}
	\end{array}
	\]
	\caption{Summary of counterexamples. We could not solve whether or not every L-shadowing map has self-tuning.}
	\label{tablaShm2}
\end{table}

Besides developing this brief theory of shadowing maps we obtained results concerning weak forms of expansivity:
\begin{itemize}
	\item In Theorem \ref{thmShmCwExpImpExp} we show that
	if $M$ is locally connected, $f$ is cw-expansive (H. Kato \cite{Kato93}) and has a shadowing map then $f$ is expansive. We say that $f$ is \textit{cw-expansive} if there is $\expc>0$ such that if $C\subset M$ is connected and $\diam(f^i(C))\leq\expc$ for all $i\in\Z$ then $C$ is a singleton

	\item In Theorem \ref{thmcwexpshbraImpcw1exp} we show that
	if $f$ is cw-expansive and has a shadowing bracket then $f$ is cw$_1$-expansive.
	Shadowing brackets are defined in \S\ref{secShBra} as continuous maps
	$[\cdot,\cdot]\colon \Delta_\delta\to M$
	satisfying: 1) $[p,p]=p$ for all $p\in M$ and
	2) for all $\epsilon>0$ there is $\gamma>0$ such that if
	$\dist(p,q)\leq\gamma$ then $[p,q]\in W^s_\epsilon(q)\cap W^u_\epsilon(p).$
	We say that $f$ is \textit{cw$_1$-expansive} if there is $\expc>0$ such that if
	$C_s,C_u\subset M$ are $\expc$-stable and $\expc$-unstable connected sets then $C_s\cap C_u$ is a singleton.
	\item In Theorem \ref{thmLeeshmN-exp} we show
	that if $f$ is $N$-expansive, has the $\mathcal{T}_0$-continuous shadowing property (\cite{Lee03}, see the definition in \S\ref{secLeeCSh}) and $M$ is a closed manifold then
	$f$ is topologically hyperbolic.
	For $N\geq 1$ $f$ is $N$-\textit{expansive} (C.A. Morales \cite{Morales2012}) if there is $\expc>0$ such that if $A\subset M$ and
	$\diam(f^i(A))\leq\expc$ for all $i\in\Z$ then $A$ has
	at most $N$ points.
\end{itemize}

Contents of the paper.
In \S\ref{secPOM} we present pseudo-orbit maps and the elementary notions we need.
In \S\ref{secShm} we introduce shadowing maps which are at the bottom of Table \ref{tablaShm}.
In \S\ref{L-shadowing maps} we consider shadowing maps related to the L-shadowing property.
The \S\ref{secAlmShInvShm} is devoted to the study of the self-tuning property.
In \S\ref{secShInvShm} we develop the shift-invariant shadowing maps.
In \S\ref{Topological hyperbolicity} we consider dynamically-invariant shadowing maps and expansivity.
In \S\ref{secLeeCSh} we prove some results concerning weak forms of expansivity and the Lee's continuous shadowing property.

\section{Pseudo-orbit maps}
\label{secPOM}
In \S\ref{secPrelimPO} we start by presenting the most elementary object of our study which are pseudo-orbits.
In \S\ref{secPseudo-orbit maps} we define pseudo-orbit maps, giving examples and proving elementary properties.
In \S\ref{secDiscrepancy functions} we study discrepancy functions to measure how far is a pseudo-orbit from being a real orbit.
In \S\ref{secCharPom} we prove an equivalence for pseudo-orbit maps.
In \S\ref{secConnecting map} we introduce the connecting map which is the basis for the brackets considered in \S\ref{secIndBracket}.

\subsection{Preliminaries on pseudo-orbits}
\label{secPrelimPO}
Let $(M,\dist)$ be a compact metric space.
Let $\tM=M^\Z$ be the set of sequences $x\colon \Z\to M$.
We denote $x(i)$ as $x_i$ for all $i\in\Z$.
The set $\tM$ is endowed with the product topology and, consequently, it is compact.
For $\mu\in(0,1)$ we consider the distances
\begin{align*}
\tdist_s(x,y)&=\sum_{i\in\Z} \mu^{|i|}\dist(x_i,y_i)\\
\tdist_m(x,y)&=\max_{i\in\Z} \mu^{|i|}\dist(x_i,y_i)
\end{align*}
for all $x,y\in \tM$.
Notice that as $M$ is compact, the set $\{\mu^{|i|}\dist(x_i,y_i):i\in\Z\}$ has indeed a maximum value.
Both distances induce the product topology on $\tM$.
On $\tM$ we consider the \textit{shift homeomorphism} $\shift\colon \tM\to \tM$
defined as
\begin{equation}
	\label{ecuDefShift}
	(\shift(x))_i=x_{i+1}\text{ for all }i\in\Z.
\end{equation}
Define $\homeo(M)$ as the set of homeomorphisms of $M$.
For $f\in\homeo(M)$ and $\delta\geq 0$ define:
\[
\tM(f,\delta)=\{x\in \tM:\dist(x_{i+1},f(x_i))\leq\delta\text{ for all } i\in\Z\}.
\]
Any element of $\tM(f,\delta)$ is called $\delta$-\textit{orbit} of $f$.
\begin{rmk}
	\label{rmkMdeltaShiftInv}
	The set $\tM(f,\delta)$ is closed and shift-invariant, \textit{i.e.}, $\shift(\tM(f,\delta))=\tM(f,\delta)$.
\end{rmk}

Define the \emph{orbit map} $\orb\colon\homeo(M)\times M\to \tM$ as
$x=\orb(f,p)=\orb_f(p)\in \tM$ where $x_i=f^i(p)$ for all $i\in\Z$.
From the definitions we have that the next diagram commutes
\[
\begin{CD}
	\tM @>\shift>> \tM\\
	@A\orb_fAA @AA\orb_fA\\
	M @>>f> M\\
\end{CD}
\]
and $\orb_f$ is injective.

\begin{rmk}
	\label{rmkOrbSube}
	This means that there is a homeomorphic copy of $M$ in $\tM$ which is the image of $\orb_f$, \textit{i.e.}, the subset of orbits of $f$. Also, $\orb_f$ is
		a dynamical conjugacy between $f$ and $\shift\colon \orb_f(M)\to\orb_f(M)$.
\end{rmk}

\subsection{Pseudo-orbit maps}
\label{secPseudo-orbit maps}
For $\delta>0$ given we consider a function
$\pom_f\colon \tM(f,\delta)\to M$.

\begin{df}
	We say that $\pom_f$ is a \textit{pseudo-orbit map}
	if
	\begin{enumerate}
		\item $\pom_f$ is continuous,
		\item $\pom_f(\orb_f(p))=p$ for all $p\in M$.
	\end{enumerate}
\end{df}
The problem is to find a connection between this elementary definition and the shadowing property.
The next definition is natural in this direction.

\begin{df}
	We say that $\pom_f$ \textit{induces shadowing} if for all $\epsilon>0$ there is $\gamma\in (0,\delta]$ such that
	if $x\in \tM(f,\gamma)$ then $\pom_f(x)$ is an $\epsilon$-shadow of $x$.
\end{df}
In \S\ref{secShm} we study when a pseudo-orbit map induces shadowing. We finish this section with some properties and examples.
\begin{prop}
	If $\pom_f$ induces shadowing\footnote{We do not assume the continuity of $\pom_f$.} then
	$\pom_f(\orb_f(p))=p$ for all $p\in M$.
\end{prop}

\begin{proof}
	Given $p\in M$ we have that $x=\orb_f(p)$ is a $\delta$-orbit for all $\delta>0$.
	Then, $\pom_f(x)$ $\epsilon$-shadows $x$ for all $\epsilon>0$.
	This means $\dist(f^i(\pom_f(x)),x_i)<\epsilon$ for all $\epsilon>0$ and $f^i(\pom_f(x))=x_i$, \textit{i.e.}
	$f^i(\pom_f(x))=f^i(p)$ and for $i=0$ we have $\pom_f(\orb_f(p))=p$.
\end{proof}

\begin{ex}
	\label{ExProjPOM}
	Perhaps, the simplest example of a pseudo-orbit map is the projection $\pom_f(x)=x_0$ defined for all $x\in \tM$.
	It is continuous because the product topology in $\tM$ makes continuous its projections.
	It is easy to see that $\pom_f(\orb_f(p))=p$, thus, the projection on the 0-coordinate is a pseudo-orbit map.
	Thus, any $f$ has a pseudo-orbit map and its existence has no dynamical consequence.
	However it is well understood when the projection induces shadowing, see \S\ref{secHsh}.
\end{ex}

In Example \ref{exaPANoShm} we will show that a pseudo-Anosov homeomorphism of the two-dimensional sphere has the shadowing property but has no pseudo-orbit map inducing shadowing.

\subsection{Discrepancy functions}
\label{secDiscrepancy functions}
This section is to introduce some techniques that will be used in
\S\ref{secCharPom} and in the proof of Proposition
\ref{propAlmSHMimpL-shm}.
Following \cite{KOP},
for $x\in \tM$ define the $f$-\textit{discrepancy} functions
\begin{align*}
	\mathcal{D}^1_f(x)  &=\sup_{i\in\Z}\dist(x_{i+1},f(x_i)),\\
	\mathcal{D}^2_f(x) &=\sup_{i\in\Z}\tdist_s(\shift^i(x),\orb_f(x_i)).\\
\end{align*}

\begin{rmk}
	It is clear that $x\in \tM$ is a $\delta$-orbit if and only if
	$\mathcal{D}^1_f(x)\leq \delta$.
\end{rmk}

The next result proves that both functions are equivalent ways to measure how far is $x\in \tM$ from being an orbit of $f$.

\begin{prop}
	\label{propDiscrep}
	The discrepancy functions satisfy:
	\begin{enumerate}
		\item for all $x\in \tM$, $\mathcal{D}^*_f(x)\geq 0$ with equality if and only if $x$ is an orbit of $f$, $*=1,2$,
		\item $\mathcal{D}^*_f(x)=\mathcal{D}^*_f(\shift(x))$ for all $x\in \tM$, $*=1,2$,
		\item $\mathcal{D}^1_f(x)\leq2\mathcal{D}^2_f(x)$ for all $x\in \tM$,
		\item for all $\epsilon>0$ there is $\delta>0$ such that
		if $\mathcal{D}^1_f(x)\leq\delta$ then
		$\mathcal{D}^2_f(x)\leq\epsilon$.
	\end{enumerate}
\end{prop}

\begin{proof}
	It is clear that they are non-negative and that $\mathcal{D}^1_f$ vanishes only at orbits. The same conclusion for $\mathcal{D}^2_f$ follows from \textit{(3)} and \textit{(4)} that we prove below.
	The proof of \textit{(2)} for $*=1$ is direct from the definition. For $*=2$ we define $y=\shift(x)$ and
	\[
	\mathcal{D}^2_f(y)=
	\sup_{i\in\Z}\tdist_s(\shift^i(y),\orb_f(y_i))
	=
	\sup_{i\in\Z}\tdist_s(\shift^{i+1}(x),\orb_f(x_{i+1}))
	=
	\mathcal{D}^2_f(x).
	\]

	To prove \textit{(3)} notice that for all $i\in\Z$ we have
	\begin{align*}
		\tdist_s(\shift^i(x),\orb_f(x_i))
		&=\sum_{j\in\Z}\frac{\dist([\shift^i(x)]_j,[\orb_f(x_i)]_j)}{2^{|j|}}\\
		&{\geq}\frac{\dist([\shift^i(x)]_1,[\orb_f(x_i)]_1)}{2} =\frac{\dist(x_{i+1},f(x_i))}{2}\\
	\end{align*}
	the inquality follows by putting $j=1$ in the sum.
	Taking $\sup_{i\in\Z}$ on both sides we obtain \textit{(3)}.

	To prove \textit{(4)} consider $\epsilon>0$ given.
	Take $n\geq 1$ such that
	$\sum_{|i|>n}\frac{\diam(M)}{2^{|i|}}<\epsilon/2$.
	Let $$\gamma=\frac{\epsilon}{2}\left(\sum_{|i|\leq n} \frac{1}{2^{|i|}}\right)^{-1}.$$
	Take $\delta>0$	such that if $x$ is a $\delta$-orbit then $\dist(f^i(x_0),x_i)<\gamma$ for $|i|\leq n$.
	Then
	\begin{align*}
		\tdist_s(x,\orb_f(x_0))&=
		\sum_{i\in\Z}\frac{\dist(x_i,f^i(x_0))}{2^{|i|}}
		=\sum_{|i|\leq n}\frac{\dist(x_i,f^i(x_0))}{2^{|i|}}+
		\sum_{|i|>n}\frac{\dist(x_i,f^i(x_0))}{2^{|i|}} \\
		&<\sum_{|i|\leq n}\frac{\gamma}{2^{|i|}}+
		\sum_{|i|>n}\frac{\diam(M)}{2^{|i|}}
		<\epsilon/2+\epsilon/2=\epsilon.
	\end{align*}
	Since $\shift^i(x)$ is also a $\delta$-orbit, we have
	\[
	\tdist_s(\shift^i(x),\orb_f((\shift^i(x))_0))=
	\tdist_s(\shift^i(x),\orb_f(x_i))<\epsilon
	\]
	for all $i\in\Z$. Taking the supremum again the proof ends.
\end{proof}

\subsection{A characterization}
\label{secCharPom}
The next characterization of pseudo-orbit maps will be used in the proof of Proposition \ref{propShmSH}.

\begin{prop}
	\label{propDigre}
	For a continuous map $\pom_f\colon \tM(f,\delta)\to M$
	the following are equivalent:
	\begin{enumerate}
		\item $\pom_f$ is a pseudo-orbit map,
		\item for all $\epsilon>0$ there is $\rho>0$ such that if
		$x\in \tM(f,\rho)$ then
		$\dist(x_i,\pom_f(\shift^i(x)))<\epsilon$ for all $i\in\Z$.
		\item for all $\epsilon>0$ there is $\rho>0$ such that if
		$x\in \tM(f,\rho)$ then
		$\dist(x_0,\pom_f(x))<\epsilon$.
	\end{enumerate}
\end{prop}

\begin{proof}
	First notice that \textit{(2)} and \textit{(3)} are equivalent because $(\shift^i(x))_0=x_i$ for all $i\in\Z$ and
	$x\in \tM(f,\rho)$ if and only if $\shift^i(x)\in \tM(f,\rho)$ for all $i\in\Z$ (Remark \ref{rmkMdeltaShiftInv}).

	To prove that \textit{(1)} implies \textit{(3)} consider $\epsilon>0$ given and take $\nu>0$ such that
	if $\tdist_s(x,y)<\nu$ then $\dist(\pom_f(x),\pom_f(y))<\epsilon$ (uniform continuity of $\pom_f$).
	By Proposition \ref{propDiscrep} there is $\rho>0$ such that if $x$ is a $\rho$-orbit then
	$\tdist_s(x,\orb_f(x_0))<\nu.$
	Therefore, given any $x\in \tM(f,\rho)$ we have
	$\dist(x_0,\pom_f(x))
	=\dist(\pom_f(\orb_f(x_0)),\pom_f(x))
	<\epsilon$.

	Let us show that \textit{(3)} implies \textit{(1)}. Given $p\in M$ let $x=\orb_f(p)$.
	Then, $x_0=p$ and $x\in \tM(f,\rho)$ for all $\rho>0$.
	Thus,
	$\dist(x_0,\pom_f(x))<\epsilon$ for $\epsilon>0$.
	This implies $x_0=\pom_f(x)$ and
	$$p=x_0=\pom_f(x)=\pom_f(\orb_f(p)).$$
	Therefore $\pom_f$ is a pseudo-orbit map.
\end{proof}

\subsection{Connecting map}
\label{secConnecting map}
Define the \textit{connecting map}
$$\con\colon \homeo(M)\times M^2\to \tM$$ as
$\con(f,p,q)=\con_f(p,q)=x$ where
\[
x_i=\left\{
\begin{array}{ll}
	f^i(p) & \text{ for } i<0, \\
	f^i(q) & \text{ for } i\geq 0.
\end{array}
\right.
\]
On $\homeo(M)$ we consider the $C^0$ distance
$$\distcc(f,g)=\sup\{\dist(f(p),g(p)), \dist(f^{-1}(p),g^{-1}(p)):p\in M\}.$$

\begin{prop}
	\label{propConCont}
	The connecting map  is continuous and uniformly continuous with respect to the variables in $M^2$:
	for all $f\in\homeo(M)$ and $\epsilon>0$ there is $\delta>0$ such that
	if $\distcc(f,g)<\delta$, $\dist(p_1,p_2)<\delta$ and $\dist(q_1,q_2)<\delta$
	then
	$$\dist(\con_f(p_1,q_1),\con_g(p_2,q_2))<\epsilon.$$
\end{prop}

\begin{proof}
	Given $\epsilon>0$ take $j\geq 0$ and $\gamma>0$ such that
	\[
	\sum_{|i|>j} \frac{\diam(M)}{2^{|i|}}<\epsilon/2
	\text{ and }
	\sum_{|i|\leq j} \frac{\gamma}{2^{|i|}}<\epsilon/2
	\]
	Given $f\in\homeo(M)$ consider $\delta>0$ such that
	if $\dist(r_1,r_2)<\delta$ and $\distcc(g,f)<\delta$
	then $\dist(f^i(r_1),g^i(r_2))<\gamma$ for all
	$|i|\leq j$.

	In this way, if $\distcc(f,g)<\delta$,
	$\dist(p_1,p_2)<\delta$, $\dist(q_1,q_2)<\delta$,
	$x=\con_f(p_1,q_1)$ and $y=\con_g(p_2,q_2)$
	then
	\[
	\dist(x,y)=\sum_{i\in\Z} \frac{\dist(x_i,y_i)}{2^{|i|}}
	=\sum_{|i|>j} \frac{\diam(M)}{2^{|i|}}
	+
	\sum_{|i|\leq j} \frac{\gamma}{2^{|i|}}<\epsilon
	\]
	and the proof ends.
\end{proof}

\begin{cor}
	\label{coroContOrb}
	The orbit map is continuous and uniformly continuous with respect to the second variable:
	for all $f\in\homeo(M)$ and $\epsilon>0$ there is $\delta>0$ such that
	if $\distcc(f,g)<\delta$ and $\dist(p,q)<\delta$ then
	$\dist(\orb_f(p),\orb_g(q))<\epsilon$.
\end{cor}

\begin{proof}
	It follows from Proposition \ref{propConCont} because $\orb_f(p)=\con_f(p,p)$ for all $f\in\homeo(M)$ and $p\in M$.
\end{proof}

\subsection{Induced brackets}
\label{secIndBracket}
Consider the $\delta$-neighborhood of the diagonal
$$\Delta_\delta=\{(p,q)\in M^2:\dist(p,q)\leq\delta\}.$$
Suppose that $\pom_f\colon \tM(f,\delta)\to M$ is a pseudo-orbit map.
Since $\con_f(\Delta_\delta)\subset \tM(f,\delta)$ we can define a bracket
$[\cdot,\cdot]\colon \Delta_\delta\to M$ as
$$[p,q]=\pom_f(\con_f(p,q)).$$
We say that this bracket is \textit{induced} by $\pom_f$.

\begin{prop}
	The induced bracket satisfies the following conditions:
	\begin{enumerate}
		\item it is continuous,
		\item $[p,p]=p$ for all $p\in M$.
	\end{enumerate}
\end{prop}
\begin{proof}
	The continuity of the bracket follows by the continuity of $\pom_f$ (definition) and $\con_f$ (Proposition \ref{propConCont}).
	For $p=q$ we have
	$
	[p,p]=\pom_f(\con_f(p,p))=\pom_f(\orb_f(p))=p.
	$
\end{proof}

\section{Shadowing maps}
\label{secShm}

In this section we introduce the shadowing maps.
In \S\ref{secShadowing maps} we prove a useful equivalence.
In \S\ref{Expansivity and shadowing maps} we show that cw-expansivity implies expansivity assuming the existence of a shadowing map and the local connection of the space.
In \S\ref{secHsh} we present known results characterizing when the projection is a shadowing map.
In \S\ref{secShBra} we define shadowing brackets. We show
that cw-expansivity with a shadowing bracket implies cw$_1$-expansivity; which in turn implies expansivity on compact surfaces.

\subsection{Shadowing maps}
\label{secShadowing maps}
The definition of shadowing map is given by considering two things. On the one hand, at least it should be a pseudo-orbit map inducing the shadowing property.
On the other hand, as explained in \S\ref{secIntro}, topologically hyperbolic homeomorphisms have shift-invariant shadowing maps.
The shift-invariance means that $\dist(f^i(\shm_f(x)),\shm_f(\shift^i(x))=0$ for all $i\in\Z$.
In Proposition \ref{propShmSH} we prove that both viewpoints coincide.
Definition \ref{dfShm}
is given as a weak form of shift-invariance.
\begin{df}
	\label{dfShm}
	A pseudo-orbit map
	$\shm_f\colon \tM(f,\delta)\to M$
	is a \textit{shadowing map} if
	for all $\epsilon>0$ there is $\gamma\in(0,\delta]$ such that if $x\in \tM(f,\gamma)$ then
	\begin{equation}
		\label{ecuShm}
		\dist(f^i(\shm_f(x)),\shm_f(\shift^i(x))\leq\epsilon\text{ for all }i\in\Z.
	\end{equation}
\end{df}
This condition is exactly what a pseudo-orbit map needs to induce the shadowing property.
\begin{prop}
	\label{propShmSH} A pseudo-orbit map induces shadowing if and only if it is a shadowing map.
	In particular, every homeomorphism with a shadowing map has the shadowing property.
\end{prop}

\begin{proof}
	(Direct). Given $\epsilon>0$ take $\gamma>0$ such that
	every $\gamma$-orbit $x$ is $\frac{\epsilon}{2}$-shadowed by $\shm_f(x)$.
	Notice that $\shift^j(x)$ also is a $\gamma$-orbit for all $j\in\Z$.
	Thus
	\[
	\dist(f^i(\shm_f(\shift^j(x))),(\shift^j(x))_i)\leq\epsilon/2
	\text{ for all }i,j\in\Z.
	\]
	For $j=0$ we have
	\[
	\dist(f^i(\shm_f(x)),x_i)\leq\epsilon/2\text{ for all }i\in\Z
	\]
	and for $i=0$ we conclude
	\[
	\dist(\shm_f(\shift^j(x)),(\shift^j(x))_0)=
	\dist(\shm_f(\shift^j(x)),x_j)\leq\epsilon/2
	\text{ for all }j\in\Z.
	\]
	The triangle inequality gives
	\[
	\begin{array}{rl}
		\dist(f^k(\shm_f(x)),\shm_f(\shift^k(x)))& \leq
		\dist(f^k(\shm_f(x)),x_k)+
		\dist(x_k,\shm_f(\shift^k(x)))\\
		&\leq \epsilon/2+\epsilon/2=\epsilon
	\end{array}
	\]
	for all $k\in\Z$. This proves \eqref{ecuShm}.

	(Converse). Given $\epsilon>0$ take $\gamma>0$ from Proposition \ref{propDigre}
	such that if $x\in \tM(f,\gamma)$ then
	$$\dist(x_i,\shm_f(\shift^i(x)))\leq\epsilon/2,\text{ for all }i\in\Z.$$
	By hypothesis, considering a smaller $\gamma$ if needed, we can also assume
	$$\dist(f^i(\shm_f(x)),\shm_f(\shift^i(x))\leq\epsilon/2\text{ for all }i\in\Z.$$
	The triangle inequality gives us
	$$\dist(f^i(\shm_f(x)),x_i)\leq\epsilon\text{ for all }i\in\Z.$$
	This proves that $\shm_f$ induces shadowing.
\end{proof}

\subsection{Cw-expansivity and shadowing maps}
\label{Expansivity and shadowing maps}

We will show that there are homeomorphisms with shadowing do not admitting a shadowing map.
For this purpose we recall the notion of continuum-wise expansivity.

A subset $C\subset M$ is a \textit{continuum} if it is closed and connected.
Following H. Kato we say that $f$ is \textit{cw-expansive} if there is $\expc>0$
such that if $C\subset M$ is a continuum and $\diam(f^i(C))\leq\expc$ for all
$i\in\Z$ then $C$ is a singleton.
If we define the \textit{dynamical ball}
\[
\Gamma_\expc(p)=\{q\in M:\dist(f^i(p),f^i(q))\leq\expc\text{ for all }i\in\Z\}
\]
then $f$ is cw-expansive if and only if there is $\expc>0$ such that
$\Gamma_\expc(p)$ is totally disconected for all $p\in M$.


\begin{thm}
	\label{thmShmCwExpImpExp}
	If $M$ is locally connected, $f$ is cw-expansive and has a shadowing map then $f$ is expansive.
\end{thm}

\begin{proof}
	Let $2\expc$ be a cw-expansivity constant and suppose that $\shm_f(x)$ is an $\expc/2$-shadow of each $\delta$-orbit $x$.
	Take $\gamma>0$, $\gamma<\delta/2,\expc/2$ such that
	\begin{equation}
		\label{ecuGamamDeltaContf}
		\text{if $\dist(a,b)<\gamma$ then $\dist(f(a),f(b))<\delta/2$.}
	\end{equation}
	As $M$ is locally connected there is $c>0$ such that if $\dist(a,b)\leq c$ then there is a continuum of diameter smaller than $\gamma$ containing $a$ and $b$.
	We will show that $c$ is an expansivity constant.

	Suppose that $\dist(f^i(p),f^i(q))\leq c$ for all $i\in\Z$.
	For each $i\in\Z$ take a continuum $C_i\subset M$ containing $f^i(p)$ and $f^i(q)$ with $\diam(C_i)<\gamma$.
	Define the set
	\[
	A=\{x\in \tM:x_i\in C_i\text{ for all } i\in\Z\}.
	\]
	Note that $\orb_f(p),\orb_f(q)\in A$.

	In this paragraph we will show that $A\subset \tM(f,\delta)$.
	If $x\in A$ and $i\in\Z$ then
	$\dist(f^i(p),x_i)<\gamma$ (because both points belong to $C_i$) and $\dist(f^{i+1}(p),f(x_{i+1}))<\delta/2$ (by \eqref{ecuGamamDeltaContf}).
	Since $x_{i+1}, f^{i+1}(p)\in C_{i+1}$ we have $\dist(x_{i+1}, f^{i+1}(p))<\gamma<\delta/2$ and the triangle inequality implies $\dist(f(x_i),x_{i+1})\leq\delta/2$.
	Thus, each $x\in A$ is a $\delta$-orbit and $A\subset \tM(f,\delta)$.

	Since $\shm_f$ is a pseudo-orbit map we conclude $\shm_f(A)$ is a continuum ($A$ is a continuum and $\shm_f$ is continuous) and $p,q\in\shm_f(A)$.

	For $r\in \shm_f(A)$ there is $x\in A$ such that $r=\shm_f(x)$.
	Since $\shm_f(x)$ is an $\expc/2$-shadow of $x$ we have
	$\dist(f^i(r),x_i)\leq\expc/2$ for all $i\in\Z$.
	Since $\dist(x_i,f^i(p))\leq\gamma\leq\expc/2$,
	by the triangle inequality we have $\dist(f^i(r),f^i(p))\leq\expc$ for all $i\in\Z$.
	This means $\shm_f(A)\subset\Gamma_\expc(p)$. Since $2\expc$ is a cw-expansivity constant the set $\Gamma_\expc(p)$ is totally disconnected. Thus, $\shm_f(A)$ is connected (proved above), totally disconnected and contains $p,q$.
	This implies that $\shm_f(A)$ is a singleton and $p=q$. This proves that $f$ is expansive with expansivity constant $c$.
\end{proof}

\begin{exa}
	\label{exaPANoShm}
	In \cite{AAV}*{Theorem 5.5} it is shown that the pseudo-Anosov homeomorphism of the two-dimensional sphere described in \cites{AAV,ArDend} is cw-expansive and has shadowing. It is well known that it is not expansive, see for instance \cite{ArDend}*{Proposition 2.2.2}.
	Therefore, it is an example with the shadowing property but without a shadowing map (it follows from Theorem \ref{thmShmCwExpImpExp} as it is not expansive).
\end{exa}

\subsection{H-shadowing}
\label{secHsh}
In this section we show examples on totally disconnected spaces having a shadowing map.
\begin{df}[Exact hit shadowing \cite{BGOR}*{Definition 3.3}]
	We say that $f$ has \textit{shadowing with exact hit}, or \textit{h-shadowing},
	if for all $\epsilon>0$ there is $\delta>0$ such that
	if $x$ is a $\delta$-orbit then $\dist(f^{i}(x_0),x_i)\leq \epsilon$ for all $i\in\Z$.
\end{df}

\begin{rmk}
	\label{rmkHshImpShm}
	From the definition we see that $f$ has h-shadowing if and only if $\shm_f(x)=x_0$ is a shadowing map.
\end{rmk}

\begin{thm}[\cites{BGO,Kurka}]
	\label{thmBGOh-sh}
	For a homeomorphism $f$ of a compact metric space $M$ the following are equivalent:
	\begin{enumerate}
		\item $f$ has h-shadowing,
		\item $f$ has shadowing and
		is equicontinuous,
		\item $f$ is equicontinuous and $M$ is totally disconnected.
	\end{enumerate}
\end{thm}

%

\begin{ex}
	Odometers are the examples of minimal and equicontinuous homeomorphisms of the Cantor set, see \cite{Kurka}*{Proposition 4.4}.
	Therefore, odometers have a shadowing map.
\end{ex}

\subsection{Shadowing bracket}
\label{secShBra}
In this section we continue \S\ref{secIndBracket} about the bracket induced by a shadowing map.
\begin{df}
	A \textit{shadowing bracket} is a continuous map $[\cdot,\cdot]\colon \Delta_\delta\to M$
	such that:
	\begin{enumerate}
		\item $[p,p]=p$ for all $p\in M$,
		\item for all $\epsilon>0$ there is $\gamma>0$ satisfying
		$\dist(p,q)\leq\gamma \Rightarrow
		[p,q]\in W^s_\epsilon(q)\cap W^u_\epsilon(p).$
	\end{enumerate}
\end{df}
Recall that the induced bracket is
$[p,q]=\shm_f(\con_f(p,q)).$
\begin{prop}
	\label{propShmShb}
	The bracket induced by a shadowing map is a shadowing bracket.
\end{prop}

\begin{proof}
	Given $\epsilon>0$ take $0<\epsilon'\leq\epsilon$ such that
	$\dist(a,b)\leq\epsilon'$, $a,b\in M$, implies $\dist(f(a),f(b))\leq\epsilon$.
	Since $f$ has shadowing there is $\gamma>0$ such that every $\gamma$-orbit $x$ is $\epsilon'$-shadowed by $\shm_f(x)$.
	For $p,q\in M$ with $\dist(p,q)\leq\gamma$ we have that $x=\con_f(p,q)$ is a $\gamma$-orbit. Thus, it is $\epsilon'$-shadowed by $r=[p,q]$.
	That is, $\dist(f^n(r),f^n(q))\leq\epsilon'\leq\epsilon$ for all $n\geq 0$.
	Also, $\dist(f^{-n}(r),f^{-n}(p))\leq\epsilon'\leq\epsilon$ for all $n\geq 1$.
	For $n=0$ we have $\dist(r,p)\leq\epsilon$ from the way we choose $\epsilon'$.
	This implies
	$r=[p,q]\in W^s_\epsilon(q)\cap W^u_\epsilon(p)$.
\end{proof}

\subsection{Cw-expansivity and shadowing brackets} In what follows we show that the pseudo-Anosov homeomorphism of the two-sphere does not admit a shadowing bracket.

A continuum is $\epsilon$-\textit{stable} if $\diam(f^n(C))\leq\epsilon$ for all
$n\geq 0$. It is $\epsilon$-\textit{unstable} if $\diam(f^{-n}(C))\leq\epsilon$ for all $n\geq 0$.
For $n\geq 1$ we say that $f$ is \textit{cw$_N$-expansive} \cite{ArDend} if there is $\expc>0$ such that if
$C_s\subset M$ is an $\expc$-stable continuum and
$C_u\subset M$ is $\expc$-unstable then $C_s\cap C_u$ contains at most $N$ points.
It is stronger than cw-expansivity.

The next proof is similar to that of Theorem \ref{thmShmCwExpImpExp}, and in both cases the topological connection is a key property. We learnt this kind of argument from J. Lewowicz, see for instance \cite{Lew89}*{Lemma 2.3}.

\begin{thm}
	\label{thmcwexpshbraImpcw1exp}
	If $f$ is cw-expansive and has a shadowing bracket then $f$ is cw$_1$-expansive.
\end{thm}

\begin{proof}
	Suppose that $2\expc>0$ is a cw-expansivity constant.
	Suppose that the bracket is defined in $\Delta_\delta$.
	Taking a smaller $\delta$ if needed we can assume that
	$[p,q]\in W^s_{\expc/3}(q)\cap W^u_{\expc/3}(p)$ whenever $\dist(p,q)\leq \delta$.

	%

	Suppose that $C_s$ and $C_u$ are $\frac{\expc}{3}$-stable/unstable continua respectively.
	Assume that $\diam(C_s\cup C_u)\leq\delta$.
	In this way we can define
	\[
	A=[C_u,C_s]=\{[p,q]:p\in C_u, q\in C_s\}.
	\]
	Note that $A$ is connected as it is the image of a connected set by a continuous function.
	If $r_1,r_2\in A$ we can write
	$r_j=[p_j,q_j]$, $p_j\in C_u$ and $q_j\in C_s$ for $j=1,2$.
	In this way we have for $i\geq 0$
	\[
	\begin{array}{rl}
		\dist(f^i(r_1),f^i(r_2))& \leq
		\dist(f^i(r_1),f^i(q_1))+
		\dist(f^i(q_1),f^i(q_2))+
		\dist(f^i(q_2),f^i(r_2))\\
		&\leq \expc/3+\expc/3+\expc/3=\expc.
	\end{array}
	\]
	Analogously, for $i<0$ we have $\dist(f^i(r_1),f^i(r_2))\leq\expc$.
	Thus, if we fix some $r\in A$ we conclude $A\subset\Gamma_{2\expc}(r)$.
	Since $2\expc$ is a cw-expansivity constant we have that $\Gamma_{2\expc}(r)$ and $A$ are totally disconnected. As $A$ is connected, it is a singleton.

	Finally notice that for each $p\in C_u\cap C_s$ we have $p=[p,p]\in A$, \textit{i.e.}, $C_u\cap C_s$ is contained in the singleton $A$ and it is a singleton or empty.
	Let us show that $\gamma=\min\{\delta/2,\expc/3\}$ is a cw$_1$-expansivity constant.
	Suppose that $C_s$ and $C_u$ are $\gamma$-stable/unstable continua respectively.
	If $C_s$ intersects $C_u$ then the diameter of their union is at most $\delta$, so we can apply our previous results to conclude that their intersection is a singleton.
\end{proof}


\begin{cor}
	If $M$ is a compact surface,
	$f$ is cw-expansive and has a shadowing bracket then $f$
	is conjugate to a linear Anosov diffeomorphism of the two-torus.
\end{cor}

\begin{proof}
	In \cite{ArDend}*{Theorem 6.8.5}
	it is shown that every cw$_1$-expansive homeomorphism of a compact surface is expansive. From Theorem
	\ref{thmcwexpshbraImpcw1exp} we conclude that $f$ is expansive.
	The shadowing bracket implies the local product structure and the shadowing property.
	Applying the classification Theorem of expansive homeomorphisms of surfaces \cites{Lew89,Hi90} the proof ends.
\end{proof}

\begin{cor}
	The pseudo-Anosov homeomorphisms of the two-sphere of Example \ref{exaPANoShm}
	has neither a shadowing bracket nor a shadowing map.
\end{cor}

\begin{proof}
	It is easy to see that this example is not cw$_1$-expansive.
	By Theorem \ref{thmcwexpshbraImpcw1exp} it admits no
	shadowing bracket.
	By Proposition \ref{propShmShb} it also admits no shadowing map.
\end{proof}

%
%
%
%

\section{L-shadowing maps}
\label{L-shadowing maps}
In this brief section we present the L-shadowing property, define L-shadowing maps, discuss some examples and prove elementary properties.
\subsection{Preliminaries on the L-shadowing property}
The L-shadowing property captures some features of expansive homeomorphisms with shadowing.
\begin{df} We say that $f$ has the \textit{L-shadowing property}
	if for all $\epsilon>0$ there is $\delta>0$ such that
	for every $\delta$-orbit $x$ that is also a two-sided limit pseudo-orbit (\textit{i.e.} $\dist(f(x_i),x_{i+1})\to 0$ as $|i|\to\infty$), there is $p\in M$ that $\epsilon$-shadows and two-sided limit shadows $x$ (\textit{i.e.} $\dist(f^i(p),x_i)\to 0$ as $|i|\to\infty$).
\end{df}

Examples of homeomorphisms with L-shadowing are structurally stable diffeomorphisms and some pseudo-Anosov maps of the two-sphere \cite{ACCV20}*{Theorem A}.
In \cite{ACCV20}*{Theorem D} it is given a homeomorphism with L-shadowing, topologically mixing and without periodic points.
We cite some known dynamical properties.

\begin{thm}[\cite{ACCV20}*{Proposition 4.1}]
	If $f$ has shadowing and its restriction to $\Omega(f)$ is expansive then $f$ has L-shadowing.
\end{thm}

\begin{thm}
	If $f$ has L-shadowing then:
	\begin{enumerate}
		\item \cite{CaCo19}*{Proposition 2} $f$ has shadowing and $R(f)$ has finitely many chain recurrent classes,
		\item \cite{ACCV20}*{Theorem B} each chain recurrent class is either expansive or contains arbitrarily small topological semi-horseshoes,
		\item \cite{ACCV20}*{Theorem F} if $C$ is a chain recurrent class of $f$ then there are $n\geq 1$, $C_1.\dots,C_n\subset C$, closed and  disjoint sets such that $C=\cup_{j=1}^n C_j$, $f(C_j)=C_{j+1\mod n}$ and $f^n$ restricted to each $C_i$ is topologically mixing.
	\end{enumerate}
\end{thm}

For $\epsilon>0$ given let $V^*_\epsilon(p)=W^*_\epsilon(p)\cap W^*(p)$ for $*=s,u$.

\begin{thm}[\cite{ACCV20}*{Theorem E}]
	\label{thmE-ACCV}
	A homeomorphism $f$ has L-shadowing if and only if it has shadowing and for all $\epsilon>0$ there is $\delta>0$ such that $\dist(p,q)<\delta$ implies $V^s_\epsilon(p)\cap V^u_\epsilon(q)\neq\emptyset$.
\end{thm}


\subsection{L-shadowing map} The next definition is natural.

\begin{df} We say that a pseudo-orbit map $\shm_f$ is an
	\textit{L-shadowing map} if for all $\epsilon>0$ there is $\gamma>0$ such that
	for every $\gamma$-orbit $x$ that is also a two-sided limit pseudo-orbit
	it holds that $\shm_f(x)$ $\epsilon$-shadows and two-sided limit shadows $x$.
\end{df}

\begin{prop}
	\label{propLshmImplicaShm}
	Every L-shadowing map is a shadowing map.
\end{prop}

\begin{proof}
	Given $\epsilon>0$ take $\gamma$ from the definition of L-shadowing map.
	For $x$ a $\gamma$-orbit define the sequence of $\gamma$-orbits $y_n$, $n\geq 1$, as
	\[
	(y_n)_i=\left\{
	\begin{array}{ll}
		x_i             & \text{ if }|i|\leq n,\\
		f^{i-n}(x_n)    & \text{ if }i\geq n,\\
		f^{i+n}(x_{-n}) & \text{ if }i\leq n.
	\end{array}
	\right.
	\]
	We see that $y_n\to x$ and each $y_n$ is a two-sided limit $\gamma$-orbit.
	By definition, $\shm_f(y_n)$ is an $\epsilon$-shadow of $y_n$.
	Taking limit we conclude that $\shm_f(x)$ is an $\epsilon$-shadow of $x$.
	Thus, $\shm_f$ induces shadowing and is a shadowing map (Proposition \ref{propShmSH}).
\end{proof}

\begin{rmk}
	\label{rmkLshmLshp}
	By definition, every L-shadowing map induces the L-shadowing property.
\end{rmk}


\subsection{L-brackets}
In light of L-shadowing maps we define L-brackets.
\begin{df}
	An \textit{L-bracket} is a continuous map $[\cdot,\cdot]\colon \Delta_\delta\to M$
	such that:
	\begin{enumerate}
		\item $[p,p]=p$ for all $p\in M$,
		\item for all $\epsilon>0$ there is $\gamma>0$ satisfying
		$$\dist(p,q)\leq\gamma \Rightarrow
		[p,q]\in V^s_\epsilon(q)\cap V^u_\epsilon(p).$$
	\end{enumerate}
\end{df}

\begin{rmk}
	From Theorem \ref{thmE-ACCV}
	we have that for a homeomorphism with an L-bracket
	the shadowing property and the L-shadowing property are equivalent.
\end{rmk}

\begin{prop}
	The bracket induced by an L-shadowing map is an L-bracket.
\end{prop}

\begin{proof}
	By Proposition \ref{propLshmImplicaShm} we know that every L-shadowing map is a shadowing map.
	Thus, we can apply Proposition \ref{propShmShb} to conclude that
	$[p,q]=\shm_f(\con_f(p,q))$ is a shadowing bracket.
	Then, given $\epsilon>0$ take $\gamma>0$ such that if
	$\dist(p,q)\leq\gamma$ then $[p,q]\in W^s_\epsilon(q)\cap W^u_\epsilon(p)$.
	Finally, if $x=\con_f(p,q)$ we see that $x$ is a two-sided limit pseudo-orbit.
	Thus, from the definition of L-shadowing map, we have $\shm_f(x)\in V^u(p)\cap V^s(q)$ and the proof ends.
\end{proof}

\begin{prop}
	\label{propOdoContraEj}
	The homeomorphisms with h-shadowing have not L-shadowing property if $M$ has infinitely many points.
	Consequently, these systems have a shadowing map (Remark \ref{rmkHshImpShm}) but not an L-shadowing map (Remark \ref{rmkLshmLshp}).
\end{prop}

\begin{proof}
	The L-shadowing property implies the existence of asymptotic points but this is incompatible with equicontinuity if $M$ has infinitely many points.
\end{proof}

\section{Self-tuning}
\label{secAlmShInvShm}
In this section we introduce and study the property of self-tuning. In \S\ref{secHyperbolic brackets} we consider hyperbolic brackets and show how they induce self-tuning pseudo-orbit maps.
It is a technical proof based on Bowen's arguments of \cite{Bowen75}.
In \S\ref{secNS} this result is applied to north-south dynamics.
We start with the definition and some fundamental properties.
\begin{df}
	\label{defAlmSHM}
	We say that the pseudo-orbit map $\shm_f\colon \tM(f,\delta)\to M$ has \textit{self-tuning} if for all $\epsilon>0$ there is $\gamma>0$ such that if
	$\tdist_s(\shift^i(x),\orb_f(x_i))<\gamma$ for some $x\in \tM(f,\delta)$ and some $i\in\Z$ then
	\begin{equation}
		\label{ecuAShInv}
		\dist(f^i(\shm_f(x)),\shm_f(\shift^i(x)))\leq\epsilon.
	\end{equation}
\end{df}

\begin{prop}
	\label{propAlmSHMimpL-shm}
	Every self-tuning pseudo-orbit map is an L-shadowing map.
\end{prop}

\begin{proof}

	Given $\epsilon>0$ take $\gamma>0$ as in Definition \ref{defAlmSHM}.
	By Proposition \ref{propDiscrep} there is $\rho>0$ such that if
	$x$ is a $\rho$-orbit then $\mathcal{D}^2_f(x)<\gamma$ and
	consequently
	$\tdist_s(\shift^i(x),\orb_f(x_i))<\gamma$ for all $i\in\Z$.
	Thus, every $\rho$-orbit satisfies \eqref{ecuAShInv} for all $i\in\Z$, $\shm_f$ is a shadowing map (Definition \ref{dfShm}) and it induces shadowing.

	Suppose that $x$ is a $\rho$-orbit that is also a two-sided limit pseudo-orbit. That is, $\dist(f(x_i),x_{i+1})\to 0$ as $|i|\to\infty$.
	This implies that $\tdist_s(\shift^i(x),\orb_f(x_i))\to 0$ as
	$|i|\to\infty$.
	Thus,
	$\dist(f^i(\shm_f(x)),\shm_f(\shift^i(x)))\to 0$ as $|i|\to\infty$.
	Also, from the uniform continuity of $\shm_f$ we have
	$$\dist(\shm_f(\shift^i(x)),x_i)
	=\dist(\shm_f(\shift^i(x)),\shm_f(\orb_f(x_i)))
	\to 0$$
	as $|i|\to\infty$.
	Finally,
	\[
	\dist(f^i(\shm_f(x)),x_i) \leq
	\dist(f^i(\shm_f(x)),\shm_f(\shift^i(x)))+
	\dist(\shm_f(\shift^i(x)),x_i)\to 0.
	\]
	This proves that $\shm_f$ is an L-shadowing map.
\end{proof}

We have neither a counterexample nor a proof for the converse of Proposition \ref{propAlmSHMimpL-shm}. The next characterization is used to prove that hyperbolic brackets induce the self-tuning property.

\begin{prop}
	\label{propCharAlmShInvShm}
	A pseudo-orbit map $\shm_f\colon \tM(f,\delta)\to M$ has self-tuning if and only if for all $\epsilon'>0$ there is $\gamma'>0$ such that if
	$\tdist_s(\shift^i(x),\orb_f(x_i))<\gamma'$ for some $x\in \tM(f,\delta)$ and some $i\in\Z$ then
	\begin{equation}
		\label{ecuAShInv2}
		\dist(f^i(\shm_f(x)),x_i)\leq\epsilon'.
	\end{equation}
\end{prop}

\begin{proof}
	We start with an argument for both directions.
	Given $\epsilon,\epsilon'>0$ take $\gamma''>0$ such that
	if $\tdist_s(y,z)<\gamma''$ then
	$\dist(\shm_f(y),\shm_f(z))<\min\{\epsilon,\epsilon'\}/2$.

	To prove the direct take $\gamma\in(0,\gamma'')$ such that if
	$\tdist_s(\shift^i(x),\orb_f(x_i))<\gamma$ then
	$$\dist(f^i(\shm_f(x)),\shm_f(\shift^i(x)))\leq\epsilon'/2.$$
	Since
	$\tdist_s(\shift^i(x),\orb_f(x_i))<\gamma''$ we have
	$$\dist(\shm_f(\shift^i(x)),x_i)=
	\dist(\shm_f(\shift^i(x)),\shm_f(\orb_f(x_i)))<\epsilon'/2.$$
	Thus, $\dist(f^i(\shm_f(x)),x_i)\leq\epsilon'$.

	The converse is quite similar.
	Take $\gamma'\in(0,\gamma'')$ such that if
	$\tdist_s(\shift^i(x),\orb_f(x_i))<\gamma'$ then
	$$\dist(f^i(\shm_f(x)),x_i)\leq\epsilon/2.$$
	Again we conclude $\dist(\shm_f(\shift^i(x)),x_i)<\epsilon/2$ and
	$\dist(f^i(\shm_f(x)),\shm_f(\shift^i(x)))\leq\epsilon$.
\end{proof}

\subsection{Hyperbolic brackets}
\label{secHyperbolic brackets}
In this section we show that hyperbolic brackets induce shadowing maps with self-tuning.
We will assume that $f$ and $f^{-1}$ are Lipschitz.

\begin{df}
	Suppose that $c\geq 1$ and $\mu\in(0,1)$.
	A bracket $[\cdot,\cdot]\colon\Delta_{\gamma}\to M$ has $(c,\mu)$-\textit{hyperbolic contraction} if $\dist(p,q)\leq\gamma$ implies
	\begin{align*}
		\dist(f^n([p,q]),f^n(q))      &\leq c\mu^n\dist(p,q),\\
		\dist(f^{-n}([p,q]),f^{-n}(p))&\leq c\mu^n\dist(p,q)
	\end{align*}
	for all $n\geq 0$.
\end{df}

\subsubsection{Lipschitz homeomorphisms}
In this section we give some elementary estimates for Lipschitz maps.
For $L>0$ and $n\geq 1$ define $L_n=1+L+L^2+\dots+L^{n-1}$. Note that
$L_1=1$ and $LL_n+1=L_{n+1}$.
Define
$$\delta_i(x)=\dist(f(x_{i-1}),x_i)$$ for all $i\in\Z$ and
$x\in \tM$.

\begin{lem}
	\label{lemShChoto}
	If $f$ is Lipschitz with Lipschitz constant $L>0$
	then
	\begin{equation}
		\label{ecuShChoto2}
		\dist(f^n(x_0),x_n)
		\leq
		\sum_{j=1}^n\delta_j(x)L^{n-j}
		\leq
		L^n\sum_{j=1}^n\delta_j(x)
		\text{ for all }n\geq 1
	\end{equation}
	for all $x\in\tM$.
	If $x$ is a $\delta$-orbit then
	\begin{equation}
	\label{ecuShChoto}
	\dist(f^n(x_0),x_n)\leq\delta L_n\text{ for all }n\geq 1.
\end{equation}
\end{lem}

\begin{proof}
	We prove the first inequality by induction.
	For $n=1$ we have
	$\dist(f(x_0),x_1)\leq\delta_1(x)$ (in fact an equality by definition).
	If the inequality is true for $n\geq 1$ then
	\begin{align*}
		\dist(f^{n+1}(x_0),x_{n+1}) &
		\leq \dist(f^{n+1}(x_0),f(x_n))+\dist(f(x_n),x_{n+1})\\
		&\leq L \sum_{j=1}^n\delta_j(x)L^{n-j} + \delta_{n+1}=\sum_{j=1}^{n+1}\delta_j(x)L^{n+1-j}
	\end{align*}
	Since $M$ is compact we have $L\geq 1$. Thus $L^n\geq L^{n-j}$ for all $1\leq j\leq n$ and
	$$\sum_{j=1}^n\delta_j(x)L^{n-j}
	\leq
	L^n\sum_{j=1}^n\delta_j(x).$$
	If $x$ is a $\delta$-orbit then $\delta_j(x)\leq\delta$ for all $j\in\Z$ and
	$$\sum_{j=1}^n\delta_j(x)L^{n-j}
	\leq
	\delta\sum_{j=1}^nL^{n-j}=\delta L_n.$$
	This proves the lemma.
\end{proof}

\subsubsection{A pseudo-orbit map}
\label{secA pseudo-orbit map}
In this section we define a pseudo-orbit map from a hyperbolic bracket.
In \S\ref{secBowenBr} we show that in fact it has self-tuning.
Suppose that $f$ is Lipschitz and has a hyperbolic bracket
$[\cdot,\cdot]\colon\Delta_{\gamma}\to M$
with parameters $c,\mu$ as explained before.
Fix $m$ large enough that
\begin{equation}
	\label{ecuBowenCondM}
	c\mu^m\leq 1/2c.
\end{equation}
For a $\delta$-orbit $x$ we try to
define the sequences $q_n,p_n\in M$, for $n\geq 0$, as
\begin{equation}
	\label{ecuBowTechDef}
	\left\{
	\begin{array}{l}
		q_0=x_0 \\
		q_n=[f^m(q_{n-1}),x_{nm}] \text{ if }n\geq 1,\\
		p_n=f^{-nm}(q_n).
	\end{array}
	\right.
\end{equation}
This makes sense if the bracket can be applied, that is, if
$\dist(f^m(q_{n-1}),x_{nm})\leq \gamma$ for all $n\geq 1$.

\begin{lem}
	\label{lemBowenWellDef}
	If $0<2\delta L_m\leq\gamma$ and $x$ is a $\delta$-orbit then
	$\dist(f^m(q_{n-1}),x_{nm})\leq 2\delta L_m$ for all $n\geq 1$.
	In particular, the sequence $q_n$ is well-defined for $\delta$-orbits with
	$0<2\delta L_m\leq\gamma$.
\end{lem}

\begin{proof}
	Arguing by induction, notice that for $n=1$ we can apply \eqref{ecuShChoto} and
	$$\dist(f^m(q_0),x_m)=
	\dist(f^m(x_0),x_m)\leq\delta L_m.$$

	Now suppose that $\dist(f^m(q_{n-1}),x_{nm})\leq2\delta L_m$ for some $n\geq 1$.
	Since $q_n=[f^m(q_{n-1}),x_{nm}]$ we can apply the hyperbolic contraction condition, the fact $c\geq 1$ and \eqref{ecuBowenCondM} to obtain
	\[
	\dist(f^m(q_n),f^m(x_{nm}))
	=\dist(f^m([f^m(q_{n-1}),x_{nm}]),f^m(x_{nm}))
	\leq c\mu^m\dist(f^m(q_{n-1}),x_{nm})\leq\frac{1}{2c}2\delta L_m\leq\delta L_m.
	\]
	Also, by \eqref{ecuShChoto} we have
	$$\dist(f^m(x_{nm}),x_{(n+1)m})=
	\dist(f^m(x_{nm}),x_{nm+m})
	\leq\delta L_m.$$
	The triangle inequality gives
	\[
	\dist(f^m(q_n),x_{(n+1)m})\leq
	\dist(f^m(q_n),f^m(x_{nm}))
	+
	\dist(f^m(x_{nm}),x_{(n+1)m})\leq
	\delta L_m+\delta L_m=2\delta L_m
	\]
	and the proof ends.
\end{proof}
In what follows we assume that $x$ is a $\delta$-orbit with
$0<2\delta L_m\leq\gamma$.

\begin{lem}
	\label{lemContHypPastLemmaBowen}
	For all $k\geq 0$, $t\geq 1$
	\begin{equation}
		\label{ecuContHypPastLemmaBowen}
		\dist(f^{-k}(q_t),f^{-k}(f^m(q_{t-1})))\leq
		c\mu^k\dist(f^m(q_{t-1}),x_{tm})\leq
		c\mu^k(2\delta L_m).
	\end{equation}
\end{lem}

\begin{proof}
	Since
	$q_t=[f^m(q_{t-1}),x_{tm}]$
	we can apply the hyperbolic contraction to obtain the first inequality.
	The last inequality follows from Lemma \ref{lemBowenWellDef}.
\end{proof}

\begin{lem}
	\label{lemConvUnifShm}
	For all $n\geq 0$ we have $\dist(p_{n+1},p_n)\leq
	c\mu^{(n+1)m}(2\delta L_m)$.
\end{lem}

\begin{proof}
	By \eqref{ecuContHypPastLemmaBowen} with $k=(n+1)m$ and $t=n+1$ gives
	\begin{align*}
		\dist(p_{n+1},p_n)&=
		\dist(f^{-(n+1)m}(q_{n+1}),f^{-nm}(q_n))\\
		&=
		\dist(f^{-(n+1)m}(q_{n+1}),f^{-(n+1)m}(f^m(q_n))
		\leq c\mu^{(n+1)m}(2\delta L_m)
	\end{align*}
	and the proof ends.
\end{proof}

For $n\geq 1$ we can define
$\shm_{f,n}^+\colon \tM(f,\delta)\to M$ as
$\shm_{f,n}^+(x)=p_n$.
From the definitions it is clear that $\shm_{f,n}^+$ is continuous.
Lemma \ref{lemConvUnifShm} means the uniform convergence of this sequence of maps.
Therefore, we can define the continuous map
$\shm_f^+=\lim_{n\to+\infty} \shm_{f,n}$.
Analogously, we can define
$\shm_f^-\colon \tM(f,\delta)\to M$ as
$\shm_f^-=\shm_{f^{-1}}^+$,
where $\shm_{f^{-1}}^+$ is the map obtained from the procedure just explained but with $f^{-1}$ in place of $f$.
Finally consider
\begin{equation}
	\label{ecuShFromBr}
	\shm_f=[\shm_f^-,\shm_f^+].
\end{equation}

\begin{prop}
	The function $\shm_f$ is a pseudo-orbit map.
\end{prop}
\begin{proof}
	It is clear that $\shm_f$ is continuous.
	In order to conclude that it is a pseudo-orbit map we need to check
	$\shm_f(\orb_f(p))=p$ for any $p\in M$.
	Let $x=\orb_f(p)$. Then $x_i=f^i(p)$ for all $i\in\Z$.
	Thus, $p_0=q_0=x_0=p$.
	For $n\geq 0$ suppose that $p_n=p$ and $q_n=f^{nm}(p)$.
	Then $q_{n+1}=[f^m(q_n),x_{(n+1)m}]=[f^m(f^{nm}(p)),f^{(n+1)m}(p)]=f^{(n+1)m}(p)$. Then $p_{n+1}=p$. By induction we proved that $p_n=p$ for all $n\geq 0$.
	Thus, $p_n\to p$ and $\shm_f^+(x)=p$.
	Analogously, $\shm_f^-(x)=p$ and
	$\shm_f(\orb_f(p))
	=\shm_f(\con_f(p,p))
	=[\shm_f^-(\con_f(p,p)),\shm_f^+(\con_f(p,p))]=[p,p]=p$.
	This proves that $\shm_f$ is a pseudo-orbit map.
\end{proof}

\subsubsection{Proving the self-tuning property}
\label{secBowenBr}
In this section we will show that the pseudo-orbit map of Equation \eqref{ecuShFromBr} has self-tuning. First we prove a series of lemmas. The difference between what we do and Bowen's estimates is that for us it is not enough to have a fixed $\delta$ for $x$, instead we need to know how much is $x$ jumping at each $i\in\Z$ as we need to \textit{tune} the shadowing where the jumps from $f(x_i)$ to $x_{i+1}$ are small.

As in Lemma \ref{lemBowenWellDef} suppose that
$0<2\delta L_m\leq\gamma$ and $x$ is a $\delta$-orbit.
Define
\[
\delta_l^m(x)=L^m\sum_{j=1}^m\delta_{(l-1)m+j}(x).
\]

\begin{lem}
	\label{lemBowenWellDef2}
	For all $u\geq 1$
	\begin{equation}
		\label{eculemBowenWellDef2}
		\dist(f^m(q_{u-1}),x_{um})
		\leq
		\sum_{l=1}^u\frac{\delta_l^m(x)}{2^{u-l}}.
	\end{equation}
\end{lem}

\begin{proof}
	Arguing by induction, notice that $u=1$ implies $l=1$ and
	we have to prove
	$$
	\dist(f^m(q_{0}),x_{m})
	\leq \delta_1^m(x)
	$$
	and this follows from \eqref{ecuShChoto2} with $u=m$. Now assume that
	\eqref{eculemBowenWellDef2} holds for some $u\geq 1$. For $u+1$ we have
	\begin{align*}
		\dist(f^m(q_u),x_{(u+1)m}) & \leq
		\dist(f^m(q_u),f^m(x_{um}))
		+
		\dist(f^m(x_{um}),x_{(u+1)m})\\
		& \leq \frac{1}{2}\dist(f^m(q_{u-1}),x_{um})
		+
		\delta_{u+1}^m(x)\\
		& \leq
		\frac12 \sum_{l=1}^u\frac{\delta_l^m(x)}{2^{u-l}}
		+
		\delta_{u+1}^m(x)
		=
		\sum_{l=1}^u\frac{\delta_l^m(x)}{2^{u+1-l}}
		+
		\frac{\delta_{u+1}^m(x)}{2^0}=\sum_{l=1}^{u+1}\frac{\delta_l^m(x)}{2^{u+1-l}}\\
	\end{align*}
	and the proof ends.
\end{proof}

\begin{lem}
	\label{lemBowenShdowing2}
	If $n\geq 1$, $i\in[0,nm]$ and
	$sm\leq i<(s+1)m$ then
	\[
	\left\{
	\begin{array}{rl}
		\dist(f^i(p_n),f^{i-sm}(q_s))         &\leq \frac{4c}{3}
			\sum_{l=1}^n\frac{\delta_l^m(x)}{2^{|l-s-1|}},\\
		\dist(f^{i-sm}(q_s),f^{i-sm}(x_{sm})) &\leq 2c \sum_{l=1}^s\frac{\delta_l^m(x)}{2^{s+1-l}},\\
		\dist(f^{i-sm}(x_{sm}),x_i)           &\leq \delta_{s+1}^m(x)
	\end{array}
	\right.
	\]
	and
	\begin{equation}
		\label{ecucotap_n}
	\dist(f^i(p_n),x_i)\leq \kappa \sum_{l=1}^\infty\frac{\delta_l^m(x)}{2^{|l-s-1|}}
	\end{equation}
	where $\kappa=\frac{10}{3}c+1$.
\end{lem}

\begin{proof}
	(First inequality). The hyperbolic contraction and Lemma \ref{lemBowenWellDef2} with $u=t$ imply
	\begin{equation}
		\label{ecuContHypPastLemmaBowen2}
		\begin{array}{rl}
			\dist(f^{-k}(q_t),f^{-k}(f^m(q_{t-1})))
			& \displaystyle\leq
			c\mu^k\dist(f^m(q_{t-1}),x_{tm})
			\leq c\mu^k
			\sum_{l=1}^t\frac{\delta_l^m(x)}{2^{t-l}}.
		\end{array}
	\end{equation}
	whenever $1\leq t\leq k$.
	The triangle inequality gives
	\begin{align*}
		\dist(f^i(p_n),f^{i-sm}(q_s))
		= \dist(f^{i-nm}(q_n),f^{i-sm}(q_s))
		&\leq \sum_{t=s+1}^n\dist(f^{i-tm}(q_t),f^{i-(t-1)m}(q_{t-1}))\\
		& = \sum_{t=s+1}^n\dist(f^{i-tm}(q_t),f^{i-tm}(f^m(q_{t-1})).
	\end{align*}
	From \eqref{ecuContHypPastLemmaBowen2} for  $k=tm-i\geq 0$ we have

	\begin{align*}
		\dist(f^{i-tm}(q_t),f^{i-tm}(f^m(q_{t-1}))
		\leq
		c\mu^{tm-i}\sum_{l=1}^t\frac{\delta_l^m(x)}{2^{t-l}}
		\leq c\mu^{(t-s-1)m}\sum_{l=1}^t\frac{\delta_l^m(x)}{2^{t-l}}
	\end{align*}
because $tm-i>(t-s-1)m$.
From \eqref{ecuBowenCondM} $\mu^m\leq 1/2$ and
$\mu^{(t-s-1)m}\leq \frac{1}{2^{t-s-1}}$.
Then
\[
\dist(f^{i-tm}(q_t),f^{i-tm}(f^m(q_{t-1}))
\leq
 c\frac{1}{2^{t-s-1}}\sum_{l=1}^t\frac{\delta_l^m(x)}{2^{t-l}}.
\]
Consequently
\[
	\dist(f^i(p_n),f^{i-sm}(q_s))
	\leq
	\sum_{t=s+1}^n\sum_{l=1}^t
	c\frac{1}{2^{t-s-1}}\frac{\delta_l^m(x)}{2^{t-l}}.
\]
We transform de double sum as follows
$$\sum_{t=s+1}^n\sum_{l=1}^t=
\sum_{t=s+1}^n\sum_{l=1}^{s}
+
\sum_{t=s+1}^n\sum_{l=s+1}^t
=
\sum_{t=s+1}^n\sum_{l=1}^{s}
+
\sum_{l=s+1}^n\sum_{t=l}^n.$$
For the first sum
	\begin{align*}
	\sum_{t=s+1}^n\sum_{l=1}^{s}\frac{1}{2^{t-s-1}}\frac{\delta_l^m(x)}{2^{t-l}}
	&=\frac{1}{2^{-s-1}}
	\sum_{t=s+1}^n\frac{1}{2^{2t}}
	\sum_{l=1}^{s}\frac{\delta_l^m(x)}{2^{-l}}\\
	&=
	\frac{1}{2^{-s-1}}
	\frac43\left(
	\frac{1}{4^{s+1}}-\frac{1}{4^{n+1}}
	\right)
	\sum_{l=1}^{s}\frac{\delta_l^m(x)}{2^{-l}}
	 < \frac{4}3
	\sum_{l=1}^{s}\frac{\delta_l^m(x)}{2^{s+1-l}}.
	\end{align*}
For the second sum
	\begin{align*}
	\sum_{t=s+1}^n
	\sum_{l=s+1}^t\frac{1}{2^{t-s-1}}\frac{\delta_l^m(x)}{2^{t-l}}
		& =
		\sum_{l=s+1}^n
		\sum_{t=l}^n
		\frac{1}{2^{t-s-1}}
		\frac{\delta_l^m(x)}{2^{t-l}}
		=
		\sum_{l=s+1}^n
		\frac{\delta_l^m(x)}{2^{-l-s-1}}
		\sum_{t=l}^n
		\frac{1}{2^{2t}}\\
		& =
		\sum_{l=s+1}^n
		\frac{\delta_l^m(x)}{2^{-l-s-1}}
		\frac43\left(\frac{1}{4^l}-\frac{1}{4^{n+1}}\right)
		<
		\frac43\sum_{l=s+1}^n
		\frac{\delta_l^m(x)}{2^{-l-s-1}}
		\left(\frac{1}{4^l}\right)\\
	 & =\frac43\sum_{l=s+1}^n\frac{\delta_l^m(x)}{2^{l-s-1}}.
	\end{align*}
	Then
	\[
	\dist(f^i(p_n),f^{i-sm}(q_s))
	\leq
	\frac{4c}{3}\sum_{l=1}^n\frac{\delta_l^m(x)}{2^{|l-s-1|}}.
	\]
	(Second inequality). Notice that
	$i-sm\geq0$ and again the hyperbolic contraction and Lemma \ref{lemBowenWellDef2} give
	\begin{align*}
		\dist(f^{i-sm}(q_s),f^{i-sm}(x_{sm}))
		& \leq c\mu^{i-sm}\dist(f^m(q_{s-1}),x_{sm})
		\leq c \sum_{l=1}^s\frac{\delta_l^m(x)}{2^{s-l}}
		\\
	\end{align*}
	(Third inequality). We will apply \eqref{ecuShChoto2}. From our hypothesis we know $0\leq i-sm<m$. Then
	\[
	\dist(f^{i-sm}(x_{sm}),x_i)
	=\dist(f^{i-sm}(x_{sm}),x_{i-sm+sm})\leq
	\sum_{j=1}^{i-sm}\delta_{sm+j}(x)L^{i-sm-j}
	\leq
	L^m\sum_{j=1}^{i-sm}\delta_{sm+j}(x)\leq\delta_{s+1}^m(x)
	\]
	(Last inequality). The inequality \eqref{ecucotap_n} follows from the triangle inequality on the first three.
\end{proof}

\begin{rmk}
	\label{rmkLipsSh}
	Notice that
	Equation \eqref{ecucotap_n}
	implies the Lipschitz shadowing, \textit{i.e.} there is $K>0$ such that every $K\epsilon$-orbit $x$ is $\epsilon$-shadowed by $\shm_f(x)$.
\end{rmk}

\begin{lem}
	\label{lemOrbPoCotas}
	If $1\leq j\leq n$ then
	\[
	\delta_j(x)\leq L\dist(f^{j-1}(x_0),x_{j-1})+\dist(f^j(x_0),x_j)
	\leq (L+1)\max_{1\leq j\leq n}\dist(f^j(x_0),x_j).
	\]
\end{lem}

\begin{proof}
	For any $j=1,\dots,n$ we have
	\begin{align*}
		\delta_j(x)=\dist(f(x_{j-1}),x_j)
		&\leq\dist(f(x_{j-1}),f^j(x_0))+\dist(f^j(x_0),x_j)		\\
		&\leq L \dist(x_{j-1},f^{j-1}(x_0))+\dist(f^j(x_0),x_j)\\
		&\leq (L+1)\max\{\dist(x_{j-1},f^{j-1}(x_0)),\dist(f^j(x_0),x_j)\}\\
		&\leq (L+1)\max_{1\leq j\leq n}\dist(f^j(x_0),x_j).
	\end{align*}
	This proves both inequalities.
\end{proof}

\begin{thm}
	\label{thmalmshinvshmbracket}
	The pseudo-orbit map induced by a hyperbolic bracket has self-tuning.
\end{thm}

\begin{proof}
We will apply Proposition \ref{propCharAlmShInvShm} for $i\geq 0$ (the case $i<0$ is analogous).
	Taking limit as $n\to+\infty$ in \eqref{ecucotap_n} we have
	\begin{equation}
		\label{ecushmas}
	\dist(f^i(\shm_f^+(x)),x_i)\leq \kappa \sum_{l=1}^\infty\frac{\delta_l^m(x)}{2^{|l-s-1|}}
	\end{equation}

	for all $i\geq 0$.
	For $\shm_f^-$ we have
	\begin{equation}
		\label{ecushmenos}
		\dist(f^{-i}(\shm_f^-(x)),x_{-i})\leq \kappa \sum_{l=1}^\infty\frac{\delta_l^m(\hat x)}{2^{|l-s-1|}}
	\end{equation}
	for all $i\geq 0$, where $\hat x$ is the pseudo-orbit defined as $\hat x_i=x_{-i}$ for all $i\in\Z$.
	By the hyperbolic contraction of the bracket we have
	\begin{align*}
		\dist(f^i(\shm_f(x)),f^i(\shm^+_f(x)))
		&\leq c\mu^i\dist(\shm^-_f(x),\shm^+_f(x))),\\
		\dist(f^{-i}(\shm_f(x)),f^{-i}(\shm^-_f(x)))
		&\leq c\mu^i\dist(\shm^-_f(x),\shm^+_f(x)))\\
	\end{align*}
	for all $i\geq 0$.
	Putting $i=0$ (and $s=0$) in Equations \eqref{ecushmas} and \eqref{ecushmenos} we have
	\begin{align*}
		\dist(\shm_f^+(x),x_0)\leq \kappa \sum_{l=1}^\infty\frac{\delta_l^m(x)}{2^{|l-1|}}\\
		\dist(\shm_f^-(x),x_0)\leq \kappa \sum_{l=1}^\infty\frac{\delta_l^m(\hat x)}{2^{|l-1|}}.
	\end{align*}
	Therefore, for $i\geq 0$
	\begin{align*}
		\dist(f^i(\shm_f(x)),x_i)
		&\leq
		\dist(f^i(\shm_f(x)),f^i(\shm_f^+(x)))+
		\dist(f^i(\shm_f^+(x)),x_i)
		\\
		&\leq
		c\mu^i\kappa
		\left(
		\sum_{l=1}^\infty\frac{\delta_l^m(x)}{2^{|l-1|}}
		+
		\sum_{l=1}^\infty\frac{\delta_l^m(\hat x)}{2^{|l-1|}}
		\right)
		+
		\kappa \sum_{l=1}^\infty\frac{\delta_l^m(x)}{2^{|l-s-1|}}
	\end{align*}
	with analogous bound for
	$\dist(f^{-i}(\shm_f(x)),x_{-i})$.

	In order to apply Proposition \ref{propCharAlmShInvShm} consider $\epsilon'>0$ given.
	Take $\hat l\geq 1$ such that
	\begin{align*}
	c\mu^i\kappa
	\left(
	\sum_{|l-1|\geq\hat l}\frac{\diam(M)}{2^{|l-1|}}
	+
	\sum_{|l-1|\geq\hat l}\frac{\diam(M)}{2^{|l-1|}}
	\right)
	+
	\kappa \sum_{|l-s-1|\geq\hat l}\frac{\diam(M)}{2^{|l-s-1|}}\\
	\leq 3c\kappa\diam(M)\sum_{|t|\geq\hat l}\frac{1}{2^{|t|}}
	\leq\epsilon'/2
	\end{align*}
	Take $\hat s$ such that if $i\geq \hat sm$ then
	\[
	c\mu^i\kappa
	\left(
	\sum_{|l-1|<\hat l}\frac{\diam(M)}{2^{|l-1|}}
	+
	\sum_{|l-1|<\hat l}\frac{\diam(M)}{2^{|l-1|}}
	\right)
	\leq \epsilon'/4.
	\]
	For $0\leq sm\leq i<\hat s m$, we can apply Lemma \ref{lemOrbPoCotas} to obtain $\gamma'$ such that if
	$\tdist_s(\shift^i(x),\orb_f(x_i))<\gamma'$ then
	\[
	c\kappa
	\left(
	\sum_{|l-1|<\hat l}\frac{\delta_l^m(x)}{2^{|l-1|}}
	+
	\sum_{|l-1|<\hat l}\frac{\delta_l^m(\hat x)}{2^{|l-1|}}
	\right)\leq \epsilon'/4.
	\]
	Also by Lemma \ref{lemOrbPoCotas} we can assume that $\gamma'$ is so small that
	\[
		\kappa \sum_{|l-s-1|<\hat l}\frac{\delta_l^m(x)}{2^{|l-s-1|}}\leq\epsilon'/4.
	\]
	Then $\dist(f^i(\shm_f(x)),x_i)\leq\epsilon'$ for all $i\geq 0$. The case $i<0$ is analogous and the proof ends.
\end{proof}

\subsubsection{Mutual inductions}
In this section we show that, after solving a technical problem, the bracket induced by the induced shadowing map is the original bracket.
That is, if
 $x=\con_f(p,q)$ for some $p,q\in M$, we consider the question: does $\shm_f(x)=[p,q]$?



\begin{prop}
	\label{propshm+}
	If $x=\con_f(p,q)$ then $\shm_f^+(x)=q$.
\end{prop}

\begin{proof}
	We have $x_n=f^n(q)$ for all $n\geq 0$.
	Thus, $p_0=q_0=x_0=q$.
	For $n\geq 0$ suppose that $p_n=q$ and $q_n=f^{nm}(q)$.
	Then $q_{n+1}=[f^m(q_n),x_{(n+1)m}]=[f^m(f^{nm}(q)),f^{(n+1)m}(q)]=f^{(n+1)m}(q)$. Then $p_{n+1}=q$. By induction we proved that $p_n=q$ for all $n\geq 0$.
	Thus, $p_n\to q$ and $\shm_f^+(x)=q$.
\end{proof}

For $\shm_f^-$ there is a problem of simmetry as the 0-th coordinate of
$\con_f(p,q)$ is $q$ and not $p$. This can be changed by considering $\sigma^{-1}(\con_f(f(p),f(q)))$, whose 0-th coordinate is $p$.
Thus, we can define $\hat\shm_f^{-}(x)=\shm_{f^{-1}}^+(\shift^{-1}(f(x)))$ and
now $\hat\shm_f^{-}(\con_f(p,q))=p$. Consequently
\[
\hat\shm_f(x)=[\hat\shm_f^{-}(x),\shm_f^{+}(x)]
\]
is a shadowing map satisfying
$\hat\shm_f(\con_f(p,q))=[p,q]$.

\begin{rmk}
	For $f$ expansive $\shift$ and $f$ commutes and $\hat\shm_f=\shm_f$.
	Also, for $f$ expansive, the shadowing map is dynamically-invariant and $\shift$-invariant, as there is a unique shadowing map.
\end{rmk}

\begin{rmk}
	If the bracket is $f$-invariant then the induced shadowing map is dynamically-invariant. This is the case if $f$ is topologically hyperbolic.
\end{rmk}

%
%

\subsection{North-south dynamics}
\label{secNS}
In this section we construct a hyperbolic bracket for a north-south dynamics to conclude that it admits an almost shift invariant shadowing map.

Let $M$ be a sphere of dimension $d\geq 1$ and take $f\colon M\to M$ a north-south homeomorphism, \textit{i.e.} there are different fixed points $N,S\in M$
and for any $p\in M\setminus\{N,S\}$ $f^n(p)\to S$ and $f^{-n}(p)\to N$ as $n\to+\infty$.
To simplify the exposition we can assume that $f$ is a diffeomorphism, $N,S$ are hyperbolic fixed points, $M$ has a Riemannian metric for which there are
$r>0$ and $\mu\in (0,1)$ such that
\begin{equation}
	\label{ecuContHipNSPolos}
	\left.
	\begin{array}{l}
		\text{if } p,q\in B_r(S)\text{ then }\dist(f^n   (p),f^n   (q))\leq \mu^n\dist(p,q),\\
		\text{if } p,q\in B_r(N)\text{ then }\dist(f^{-n}(p),f^{-n}(q))\leq \mu^n\dist(p,q),
	\end{array}
	\right\} \text{ for all }n\geq 0.
\end{equation}
We assume $r$ so that $B_r(N)$ and $B_r(S)$ have disjoint closures.
Let $L>0$ be a Lipschitz constant for $f$ and $f^{-1}$.

To define the bracket consider a map $\varphi\colon M\to I=\{t\in\R:0\leq t\leq 1\}$\footnote{We do not use the standard notation $I=[0,1]$ to avoid confusion with the brackets we are considering.} such that
\begin{align*}
	\varphi(p)=1&\text{ if }p\in B_r(N),\\
	\varphi(p)=0&\text{ if }p\in B_r(S).
\end{align*}
Let $\rho>0$ be such that the exponential map $\exp_p\colon T_pM\to M$ is injective on $\{v\in T_pM:\|v\|\leq\rho\}$ for all $p\in M$.
If $\dist(p,q)\leq\rho$ define
$$[p,q]=\exp_p(\varphi(p)\exp_p^{-1}(q)).$$
The bracket chooses a point in the shortest geodesic segment from $p$ to $q$.
Notice that
\begin{align*}
	[p,q]=q &\text{ if }p\in B_r(N),\\
	[p,q]=p &\text{ if }p\in B_r(S),\\
\end{align*}
For $p,q\in B_r(N)$ we have
\begin{align*}
	\dist(f^{n}(q),f^{n}([p,q]))&=	\dist(f^{n}(q),f^{n}(q))=0\leq \mu^n\dist(p,q),\\
	\dist(f^{-n}(p),f^{-n}([p,q]))&=	\dist(f^{-n}(p),f^{-n}(q))\leq \mu^n\dist(p,q),\\
\end{align*}
for all $n\geq 0$.
Analogously, for $p,q\in B_r(S)$ we have
\begin{align*}
	\dist(f^{n}(q),f^{n}([p,q]))&=	\dist(f^{n}(q),f^{n}(p))\leq \mu^n\dist(p,q),\\
	\dist(f^{-n}(p),f^{-n}([p,q]))&=	\dist(f^{-n}(p),f^{-n}(p))=0\leq \mu^n\dist(p,q),\\
\end{align*}
for all $n\geq 0$.
Take $\delta>0$, smaller than $r$ and $\rho$ such that
if $$U=\{p\in M:\dist(p,N)>r-\delta\text{ and }\dist(p,S)>r-\delta\}$$
then $p,q\in U$ and $\dist(p,q)<\delta$ implies that the geodesic from $p$ to $q$ (\textit{via} the injective restriction of the exponential map) is disjoint from the $r/2$-balls centered at $N$ and $S$.
In this way for $p,q$ not in $B_r(N)$ nor in $B_r(S)$ but with $\dist(p,q)<\delta$ we have that:
\begin{itemize}
	\item $[p,q]$ is not in $B_{r/2}(N)\cup B_{r/2}(S)$;
	\item thus, we can take $u\geq 1$ such that
	$f^i(p),f^i(q),f^i([p,q])\in B_{r/2}(N)\cup B_{r/2}(S)$ whenever $|i|\geq u$;
	\item also, we can assume that
	\begin{equation}
		\label{ecuUNS}
		f^u(B_r(N))\cup f^{-u}(B_r(S))=M.
	\end{equation}
	This condition is used in the proof of Proposition \ref{propNoShInvNS};
	\item consequently
	\begin{equation}
		\left.
		\begin{array}{r}
			\dist(f^{n} (q),f^{n} ([p,q]))\leq (L^u/\mu^u) \mu^n\dist(p,q)\\
			\dist(f^{-n}(p),f^{-n}([p,q]))\leq (L^u/\mu^u) \mu^n\dist(p,q)\\
		\end{array}
		\right\}\text{ for all }n\geq 0.
	\end{equation}
\end{itemize}
Taking
\begin{equation}
	\label{ecuC'CorchNS}
	c=L^u/\mu^u
\end{equation}
we have that $[\cdot,\cdot]$ is a hyperbolic bracket.

\begin{rmk}
	This bracket cannot be $f$-invariant because by Proposition \ref{propExpDesdeCorchete} $f$ should have to be expansive.
\end{rmk}


\begin{que}
Does every structurally stable $C^1$ diffeomorphism
admit a hyperbolic bracket?
\end{que}

\begin{prop}
	\label{propNoShInvNS}
	The shadowing map
induced by the hyperbolic bracket for the north-south dynamics has self-tuning but is not shift-invariant.
\end{prop}

\begin{proof}
The shadowing map has self-tuning by Theorem \ref{thmalmshinvshmbracket}.
From Equation \eqref{ecuUNS} we know $f^u(B_r(N))\cup f^{-u}(B_r(S))=M$, thus for each $p\in M$ there is at most one integer $i$ such that $\varphi(f^{iu}(p))\in(0,1)$.
	Take $p,q\in M\setminus (B_r(N)\cup B_r(S))$ with $\dist(p,q)\leq \delta$, $p\neq q$ and $\varphi(p)\in(0,1)$.
	Let $x=\con_f(p,q)$.
	In this way, $\shm_f(x)\neq q$.
For $\shift^u(x)$
	we have $(\shift^u(x))_i=x_{u+i}=f^{u+i}(q)$.
	Thus, as in the proof of Proposition \ref{propshm+}, we have $\shm_f^+(x)=f^u(q)$.
Since $f^u(q)\in B_r(S)$ if $\delta$ is small enough we have
$\shm^{-}_f(\shift^u(x))\in B_r(S)$ and
$\varphi(\shm^{-}_f(\shift^u(x)))=1$.
This implies $[\shm^{-}_f(\shift^u(x)),\shm^{+}_f(\shift^u(x))]=f^u(q)$ and
$\shm_f(\shift^u(x))=f^u(q)$. We see from this that $\shm_f(\sigma^u(x))\neq f^u(\shm_f(x))$.
\end{proof}

\begin{que}
We do not know whether or not a north-south dynamics admits (another) shadowing map being shift-invariant.
\end{que}

For an arbitrary shadowing map we can say the following.

\begin{prop}
	Every shadowing map for a north-south dynamics is an
	L-shadowing map.
\end{prop}

\begin{proof}
	Let $x$ be a two-sided limit $\delta$-orbit.
	We will show that $x_n$ congerges to $N$ or $S$ as $n\to+\infty$.
	From \eqref{ecuContHipNSPolos} we have that
	if $\dist(x_{n_0},S)\leq \frac{\delta}{1-\mu}$ then
	$\dist(x_{n},S)\leq \frac{\delta}{1-\mu}$ for all $n\geq n_0$.
	If $x_n$ does not converge to $N$, then there are $\epsilon>0$ and $n_k\to\infty$, $k\geq 1$, such that $\dist(x_{n_k},N)>\epsilon$ for all $k\geq 1$.
	As $x$ is two sided-limit we see that after $x_{n_k}$, for $k$ large,
	$x_{n_k+j}\in B_{\delta/(1-\mu)}(S)$ and
	$x_{n}\in B_{\delta/(1-\mu)}(S)$ for all $n\geq n_k+j$.
	This proves that $x_n\to S$.
	Since the orbits of $f$ also converges to $N$ or $S$, we see that any shadowing map is L-shadowing.
\end{proof}

\begin{rmk}
	\label{rmkKOP}
	The following definition was given in \cite{KOP} for Lipschitz maps (not necessarily invertible) of compact subsets of an Euclidean space.
	For $n\geq 1$ let
	\[
	M^n(f,\delta)=\{x\in\tM:\dist(f^j(x_{j-1}),x_j)\leq\delta\text{ for all }1\leq j\leq n\}.
	\]
	We say that $f$ is \textit{continuously shadowing} with positive parameters $k,\delta>0$
	if for each $n\geq 1$ there is a continuous function
	\[
	W^n\colon M^n(f,\delta)\to M
	\]
	such that
	\begin{equation}
		\label{ecuContShKop}
		\sup_{0\leq j\leq n}\dist(f^j(W^n(x)),x_j) \leq k \sup_{1\leq j\leq n}\dist(f^j(x_{j-1}),x_j).
	\end{equation}
	In \cite{KOP}*{Theorem 1} it is proved that
	if $f$ is a semi-hyperbolic map then it is continuously shadowing.
	This approach ir related to the Lipschitz shadowing observed in Remark \ref{rmkLipsSh}.
\end{rmk}

%
%
%

\section{Shift invariance}
\label{secShInvShm}
In this section we study shift-invariant shadowing maps.
In \S\ref{secElShift} we develop the shift dynamical system and show that it has a shift-invariant shadowing map.\footnote{In this sentence the word 'shift' appears with two different meanings. On the one hand, we consider it as a dynamical system, \textit{i.e.} $\shift$ takes the place of $f$.
On the other hand, the 'shift-invariance' is a property of shadowing maps for homeomorphisms; in this case the homeomorphism is a shift.}
In \S\ref{secTopoEst} we show that shift-invariance implies Walter's topological stability.
In \S\ref{secRecurrence} we prove some dynamical consequences concerning
chain recurrence, periodic points and the non-wandering set.
In \S\ref{secBracket with uniform contraction} we show that shift-invariant shadowing maps induce brackets with uniform contraction.

We start with the definition and remarks.

\begin{df}
	A pseudo-orbit map is \textit{shift-invariant} if for some $\delta>0$ the next diagram commutes
	\begin{equation}
		\label{ecuEquivariant}
		\begin{CD}
			\tM(f,\delta) @>\shift>> \tM(f,\delta)\\
			@V\shm_f VV @VV\shm_f V\\
			M @>>f> M\\
		\end{CD}.
	\end{equation}
\end{df}

From this definition we see that
any shift-invariant pseudo-orbit map is a semiconjugacy of the shift map and $f$.

\begin{rmk}
	\label{rmkshInvImpAShInv}
	From the definitions we also see that every shift-invariant pseudo-oprbit map has self-tuning.
\end{rmk}

\subsection{The shift homeomorphism}
\label{secElShift}
The shift homeomorphism is a well-known dynamical system and in this section we will exploit its dynamical properties.

\subsubsection{Bracket}
\label{secBracketShift}
As before, given a compact metric space $M$ consider $\tM=M^\Z$.
Now $\tM$ will be the phase space of our dynamical system (the shift).
Define $[\cdot,\cdot]\colon \tM\times \tM\to \tM$ as
\[
 [x,y]_i=
 \left\{
 \begin{array}{cl}
  x_i & \text{ if } i\leq -1,\\
  y_i & \text{ if } i\geq 0.
 \end{array}
 \right.
\]

\begin{prop}
 This bracket satisfies:
 \begin{enumerate}
  \item it is continuous,
  \item $[x,x]=x$ for all $x\in \tM$,
  \item $[x,[y,z]]=[x,z]=[[x,y],z]$ for all $x,y,z\in \tM$,
  \item $\tdist_s(x,[x,y])+\tdist_s([x,y],y)=\tdist_s(x,y)$ for all $x,y\in\tM$,
  \item $\max\{\tdist_m(x,[x,y]),\tdist_m([x,y],y)\}=\tdist_m(x,y)$ for all $x,y\in\tM$.
 \end{enumerate}
\end{prop}

The proof of this proposition is direct from definitions.
As explained in \cite{Ruelle}*{\S 7.1} the first three conditions implies a \textit{global product structure} as follows. For $x\in\tM$ define
\begin{align*}
 V^-(x)=\{[x,y]\in \tM :y\in\tM\},\\
 V^+(x)=\{[y,x]\in \tM :y\in\tM\},
\end{align*}
then
\[
[\cdot,\cdot]\colon V^+(x)\times V^-(x)\to \tM
\]
is a homeomorphism, for any $x\in\tM$.

\subsubsection{The shift map}
\label{secPropDynShift}
The shift homeomorphism $\shift$ was already defined in \eqref{ecuDefShift}. As a dynamical system it has some remarkable properties:
\begin{enumerate}
 \item it is topologically mixing,
 \item it is expansive if and only if $M$ is finite,
 \item it has the shadowing property,
 \item the set of periodic points is dense in $\tM$,
 \item for $y,z\in V^+(x)$ we have $\tdist_s(\shift^n(y),\shift^n(z))=\mu^n \tdist_s(y,z)$ for all $n\geq 0$,
 \item for $y,z\in V^-(x)$ we have $\tdist_s(\shift^{-n}(y),\shift^{-n}(z))=\mu^n \tdist_s(y,z)$ for all $n\geq 0$.
\end{enumerate}
In the last items we use the parameter $\mu\in(0,1)$ that we consiedered in \S\ref{secPrelimPO} to define the metrics of $\tM$.
\subsubsection{A shadowing map}
A proof of the shadowing property of the shift homeomorphism can be found in \cite{AH}*{Theorem 2.3.12}.
There, essentially, it is defined a function
$\shm_{\shift}\colon \tM^\Z\to\tM$ as
$$(\shm_{\shift}(\alpha))_i=(\alpha_i)_0.$$
On $\tM^\Z$ we consider the product topology and the
distance $\ttdist_s$ define as follows.
For $\alpha,\beta\in \tM^\Z$
\[
 \ttdist_s(\alpha,\beta)
 =\sum_{i\in\Z}\mu^{|i|}\tdist_s(\alpha_i,\beta_i)
 =\sum_{i\in\Z}\mu^{|i|}\sum_{j\in\Z}\mu^{|j|}\dist((\alpha_i)_j,(\beta_i)_j)
 =\sum_{i,j\in\Z}\mu^{|i|+|j|}\dist((\alpha_i)_j,(\beta_i)_j).
\]
Notice that
\[
 \tdist_s(\shm_{\shift}(\alpha),(\shm_{\shift}(\beta)))
 =\sum_{i\in\Z}\mu^{|i|}\dist((\shm_{\shift}(\alpha))_i,(\shm_{\shift}(\beta))_i)
 =\sum_{i\in\Z}\mu^{|i|}\dist((\alpha_i)_0,(\beta_i)_0).
\]
Therefore, $\tdist_s(\shm_{\shift}(\alpha),\shm_{\shift}(\beta))\leq\ttdist_s(\alpha,\beta)$. In particular this proves the continuity of $\shm_{\shift}$.
Define $\orb_{\shift}\colon \tM\to\tM^\Z$ as
$(\orb_{\shift}(x))_i=\shift^i(x)$ for all $i\in\Z$.
It satisfies $\shm_{\shift}(\orb_{\shift}(x))=x$ for all $x\in \tM$. Indeed, for $i\in\Z$
\[
 (\shm_{\shift}(\orb_{\shift}(x)))_i=
 ((\orb_{\shift}(x))_i)_0
 =(\shift^i(x))_0
 =x_i.
\]
Let $\tilde\shift$ be the shift homeomorphism of $\tM^\Z$.
That is, $\tilde\shift\colon\tM^\Z\to\tM^\Z$ is defined as
$$(\tshift(\alpha))_i=\alpha_{i+1}$$ for all $i\in\Z$ (just as for $\shift$).
\begin{prop}
	\label{propShInvShmParaElShift}
	The shadowing map $\shm_\sigma$ is shift-invariant, \textit{i.e.},
	the diagram commutes
	\[
	\begin{CD}
		\tM^\Z @>\tilde\shift>> \tM^\Z\\
		@V\shm_{\shift} VV @VV\shm_{\shift} V\\
		\tM @>>\shift> \tM\\
	\end{CD}.
	\]
\end{prop}
\begin{proof}
	For $\alpha\in \tM^\Z$ we have
	\begin{align*}
		(\shm_{\shift}(\tshift(\alpha)))_i=
		((\tshift(\alpha))_i)_0
		&=(\alpha_{i+1})_0\text{ and }\\
		(\shift(\shm_{\shift}(\alpha)))_i=
		(\shm_{\shift}(\alpha))_{i+1}&=(\alpha_{i+1})_0
	\end{align*}
	which proves the result.
\end{proof}

\begin{rmk}
	The notation $\tM$ for $M^\Z$ has the purpose of hidding
	that $\tM^\Z$ has two natural shift actions. Indeed, $\tM^\Z=(M^\Z)^\Z$ and this can be regarded as $M^{\Z\times\Z}$ and
	then the shift can act on any coordinate.
	If $\alpha\in\tM^\Z$ we can write $(\alpha_i)_j=\alpha_{(i,j)}$. The shift $\tshift$ acts on the first variable, adds 1 to $i$.
	The other shift add 1 to $j$ and is equivalent to allow $\shift$ to act on each coordinate of $\tM^\Z$.
	That is, abusing of the notation, we consider
	$\shift\colon \tM^\Z\to\tM^\Z$ as
	$(\shift(\alpha))_i=\shift(\alpha_i)$.
	We see that
	\begin{align*}
		((\shift(\alpha))_i)_j=(\shift(\alpha_i))_j&=(\alpha_i)_{j+1}\text{ while }\\
		((\tshift(\alpha))_i)_j&=(\alpha_{i+1})_j.
	\end{align*}
\end{rmk}

It is natural to ask for the commutativity of the diagram ($\shift$ in place of $\tshift$)
 \[
\begin{CD}
 \tM^\Z @>\shift>> \tM^\Z\\
 @V\shm_{\shift} VV @VV\shm_{\shift} V\\
 \tM @>>\shift> \tM.\\
\end{CD}
\]
\begin{rmk}
	This diagram commutes if and only if $M$ is a singleton.
\end{rmk}
\begin{proof}
	If $M$ is a singleton all these spaces are singletons and any diagram commutes.
	To see the converse, notice that we already know that
	$(\shift(\shm_{\shift}(\alpha)))_i=(\alpha_{i+1})_0$.
	Also
	$$
	\shm_{\shift}(\shift(\alpha))_i
	=((\shift(\alpha))_i)_0
	=(\shift(\alpha_i))_0
	=(\alpha_i)_1.
	$$
	Thus, if $M$ has two different points we can take $\alpha$ with
	$(\alpha_i)_1\neq (\alpha_{i+1})_0$, for some $i\in\Z$, and the diagram does not commute.
\end{proof}

To illustrate a result we prove later (Theorem \ref{thmRuelleBExp}, the expansivity for homeomorphisms with a Ruelle's bracket), for $\delta>0$ given define
\[
 \tM^\Z(\shift,\delta)=\{\alpha\in\tM^\Z:\tdist_s(\shift(\alpha_i),\alpha_{i+1})\leq\delta\text{ for all }i\in\Z\}.
\]
This is the set of $\delta$-orbits of $\shift$.
\footnote{In general this set is not invariant by $\shift$, but as $\shift$ is continuous for all $\epsilon>0$ there is $\delta>0$ such that
if $\alpha$ is a $\delta$-orbit then $\shift(\alpha)$ is an $\epsilon$-orbit.}

\begin{prop}
	\label{propShifExpEquiv}
	The following statements are equivalent:
	\begin{enumerate}
		\item The diagram
    \[
    \begin{CD}
      \tM^\Z(\shift,\delta) @>\shift>> \tM^\Z\\
      @V\shm_{\shift} VV @VV\shm_{\shift} V\\
      \tM @>>\shift> \tM\\
    \end{CD}.
    \]
    commutes for some $\delta$,
    \item $M$ is finite,
    \item $\sigma$ is expansive.
	\end{enumerate}
\end{prop}

\begin{proof}
  We have already commented the equivalence of (2) and (3).
  We have seen that this diagram commutes for some $\alpha$ if and only if $(\alpha_i)_1= (\alpha_{i+1})_0$ for all $i\in\Z$.
  If $M$ is finite we can take $\delta$ such that $\dist(p,q)\leq\delta$ implies $p=q$. If $\alpha\in\tM^\Z(\shift,\delta)$ then
  \[
  \tdist_s(\shift(\alpha_i),\alpha_{i+1})=
  \sum_{j\in\Z}\mu^{|j|}\dist((\shift(\alpha_i))_j,(\alpha_{i+1})_j)=
  \sum_{j\in\Z}\mu^{|j|}\dist((\alpha_i)_{j+1},(\alpha_{i+1})_j)\leq\delta\text{ for all } i\in\Z.
  \]
  For $j=0$ we have
  $\dist((\alpha_i)_{1},(\alpha_{i+1})_0)\leq\delta$ for all $i\in\Z$.
  Thus $(\alpha_i)_{1}=(\alpha_{i+1})_0$ for all $i\in\Z$ and the diagram commutes.

  If $M$ has infinitely many points, given any $\delta>0$ we can take $p,q$ arbitrarily close such that if we define $(\alpha_i)_j=p$ for all $i,j\in\Z$ except for $(\alpha_1)_0=q$ then
  $\alpha\in\tM^\Z(\shift,\delta)$ but $(\alpha_1)_0=q\neq p=(\alpha_0)_1$ and the diagram does not commute.
\end{proof}

\begin{rmk}
	\label{rmkShiftShiftInvShmNoExp}
	If $M$ has infinitely many points then $\shift\colon \tM\to\tM$ is not expansive (Proposition \ref{propShifExpEquiv}) but has a shift-invariant shadowing map (Proposition \ref{propShInvShmParaElShift}).
\end{rmk}

\subsubsection{Mutual inductions}
Shadowing maps always induce brackets and some brackets induce shadowing maps (for instance, with Bowen's technique \S \ref{secBowenBr}).
We will ilustrate this for the shift homeomorphism (for the shadowing map and the bracket already defined).
Define $\con_{\sigma}\colon \tM\times\tM\to\tM^\Z$ as
\[
	\con_\sigma(x,y)_i=
	\left\{
	\begin{array}{ll}
		\sigma^i(y)	& \text{ if }i\geq 0,\\
		\sigma^i(x)	& \text{ if }i<0.
	\end{array}
	\right.
\]
For $x,y\in\tM$ we have
$$
	[x,y]=\shm_\sigma(\con_\sigma(x,y))
$$
and we see how a shadowing map induces a bracket.

Conversely, we now start with the bracket.
Take $\alpha\in\tM^\Z$.
As in \S\ref{secA pseudo-orbit map} define sequences $x_n,y_n\in\tM$, for $n\geq 0$ as
\[
	\left\{
	\begin{array}{l}
		x_0=\alpha_0,\\
		x_n=[\sigma(x_{n-1}),\alpha_n], n\geq 1,\\
		y_n=\sigma^{-n}(x_n).
	\end{array}
	\right.
\]
The next result formally explains the beautiful work of these equations.

\begin{prop}
	For all $n\geq 0$ and $0\leq i\leq n$ we have
	$(y_n)_i=(\alpha_i)_0$.
\end{prop}

\begin{proof}
For $n=0$, $i=0$ we have $(y_0)_0=(x_0)_0=(\alpha_0)_0$.
Take $n\geq 1$ and assume that
$(y_n)_i=(\alpha_i)_0$ for $0\leq i\leq n$.
If $0\leq i\leq n$ then
\[
(y_{n+1})_i=(\shift^{-n-1}(x_{n+1}))_i=
(x_{n+1})_{i-n-1}=
[\sigma(x_{n}),\alpha_{n+1}]_{i-n-1}=
(\sigma(x_{n}))_{i-n-1}=
(x_{n})_{i-n}\text{ and }
\]
\[
	(\alpha_i)_0=(y_n)_i=
	(\sigma^{-n}(x_n))_i=
	(x_n)_{i-n}.
\]
Thus $(y_{n+1})_i=(\alpha_i)_0$ for $0\leq i\leq n$.
For $i=n+1$
\[
(y_{n+1})_{n+1}=
(\shift^{-n-1}(x_{n+1}))_{n+1}=
(x_{n+1})_{0}=
[\sigma(x_{n}),\alpha_{n+1}]_{0}=(\alpha_{n+1})_{0}
\]
and the proof ends.
\end{proof}

From this we see that $y_n$ is a convergent sequence and its limit $z^+\in\tM$ satisfies $z^+_n=(\alpha_n)_0$ for all $n\geq 0$.
Analogously, with $\sigma^{-1}$ in place of $\sigma$ we obtain $z^-\in\tM$
with $z^-_{-n}=(\alpha_{-n})_0$ for all $n\geq 0$.
Finally, the shadowing map is recovered by
$\shm_{\shift}(\alpha)=[z^-,z^+]$.
This way of inducing a shadowing map from a bracket is analogous to Bowen's technique shown in \S\ref{secBowenBr}.

\subsection{Topological stability}
\label{secTopoEst}
Following Walters \cites{Walters1970,Walters1978} we say that $f\in\homeo(M)$ is \textit{topologically stable} if
for all $\epsilon>0$ there is $\delta>0$ such that
if $\dist_{C^0}(f,g)<\delta$ then there is a continuous map
$h\colon M\to M$ such that
$\dist(h,I_M)<\epsilon$ and $fh=hg$ as in the next diagram
\[
\begin{CD}
	M @>g>> M\\
	@VhVV @VVhV\\
	M @>>f> M\\
\end{CD}
\]
\begin{rmk}
	In \cite{Ni71} Z. Nitecki proved that structurally stable $C^1$ diffeomorphisms are topologically stable.
	Notice that topological stability is invariant under conjugacy, thus any homeomorphism conjugate to a structurally stable diffeomorphism is topologically stable. These are the only examples of topologically stable homeomorphisms on closed manifolds known by the author.
	Walters \cite{Walters1978} proved that topologically stable homeomorphisms of closed manifold have the shadowing property.
\end{rmk}

\begin{prop}
	If $f$ admits a shift-invariant shadowing map then it is topologically stable.
\end{prop}

\begin{proof}
	Let $\shm_f\colon \tM(f,\beta)\to M$ be a shadowing map for $f$ and take $\epsilon>0$.
	By the continuity of $\shm_f$ there is $\gamma>0$ such that $\dist(x,y)<\gamma$ implies
	$\dist(\shm_f(x),\shm_f(y))<\epsilon$ for all $x,y\in \tM(f,\beta)$.
	By Corollary \ref{coroContOrb} there is $\delta>0$ such that
	$\distcc(f,g)<\delta$ implies $\dist(\orb_g(p),\orb_f(p))<\gamma$ for all $p\in M$.

	Suppose that $\dist_{C^0}(f,g)<\delta$ and
	define $h\colon M\to M$ as
	$h(p)=\shm_f(\orb_g(p))$.
	We have
	\[
	\begin{array}{rl}
		\distcc(h,Id_M) & \displaystyle =\sup_{p\in M}\dist(h(p),p) =\sup_{p\in M}\dist(\shm_f(\orb_g(p)),p)\\
		& \displaystyle =\sup_{p\in M}\dist(\shm_f(\orb_g(p)),\shm_f(\orb_f(p)))<\epsilon.
	\end{array}
	\]
	We know that $h$ is continuous by the continuity of $\shm_f$ and Corollary \ref{coroContOrb}.
	We have that $fh=hg$ because $\shm_f$ satisfies \eqref{ecuEquivariant}. We proved that $f$ is topologically stable.
\end{proof}

\subsection{Recurrence}
\label{secRecurrence}
Given $p,q\in M$ and $\alpha>0$ we write $p\overset\alpha\sim q$ if there are $\alpha$-orbits of $f$ such that $x_0=p,x_1,\dots,x_k=q$ and $q=y_0,y_1,\dots,y_l=p$.
If $p\overset\alpha\sim q$ for all $\alpha>0$ we write $p\sim q$.
The \textit{chain recurrence set} of $f$ is defined as
\[
R(f)=\{p\in M:p\sim p\}.
\]
A point $p\in M$ is \textit{wandering} if it has a neighborhood $U$ such that $f^n(U)\cap U=\emptyset$ for all $n\geq 1$. The set of non-wandering points is denoted as $\Omega(f)$.
In general $\Omega(f)\subset R(f)$ and the equality holds if $f$ has shadowing.
A $\delta$-orbit $x$ is \textit{periodic} if there is $n\geq 1$ such that $x_{i+n}=x_i$ for all $i\in\Z$ (\textit{i.e.} it is periodic by the shift homeomorphism).
In \cite{Aoki83} N. Aoki proved that if $f$ has the shadowing property then the restriction of $f$ to the non-wandering set has shadowing too.

\begin{prop}
	If $f$ admits a shift-invariant shadowing map $\shm_f\colon \tM(f,\delta)\to M$ then:
	\begin{enumerate}[label=(\alph*)]

		\item\label{thmShmSH-R} $R(f)=\clos(\per(f))$.
		\item\label{thmShmSH-Omega} $f|_{\Omega(f)}$ has a shift-invariant shadowing map.
	\end{enumerate}
\end{prop}

\begin{proof}
	\textit{\ref{thmShmSH-R}.} Take $p\in R(f)$ and $\epsilon>0$.
	Take $\gamma>0$ such that if
	$x\in \tM(f,\gamma)$ then $\shm_f(x)$ is an $\epsilon$-shadow of $x$.
	As $p$ is chain recurrent, there is a periodic $\gamma$-orbit $x$ with $x_0=p$.
	If $n$ is the period of $x$ (by the shift) then
	$\shm_f(x)=\shm_f(\shift^n(x))=f^n(\shm_f(x))$, thus $\shm_f(x)$ is a periodic point.
	Since $q=\shm_f(x)$ is an $\epsilon$-shadow of $x$ we have that $\dist(q,x_0)=\dist(q,p)<\epsilon$ and $q$ is periodic for $f$. We proved that $R(f)\subset\clos(\per(f))$. As the other inclusion holds for any homeomorphism the proof ends.

	\textit{\ref{thmShmSH-Omega}.}
	For this proof we follow \cite{Aoki83} (in fact with our hypothesis the argument is simpler).
	As explained in that paper, there are finitely many equivalence classes of $\overset\delta\sim$ and each class is closed and $f$-invariant.
	Let $A_1,\dots,A_k$ be the equivalence classes of $\overset\delta\sim$.
	Take $\delta_1=\inf\{\dist(A_i,A_j):1\leq i<j\leq k\}>0$ and
	$\alpha>0$ such that $\alpha<\min\{\delta,\delta_1\}$.
	Define
	\[
	\Omega^\Z(f,\alpha)=\{x\in \tM(f,\alpha): x_i\in\Omega(f) \text{ for all }i\in\Z\}.
	\]

	We will show that $\shm_f(x)\in\Omega(f)$ for all $x\in\Omega^\Z(f,\alpha)$.
	For $x\in\Omega^\Z(f,\alpha)$
	notice that there is a class $A_j$ containing $x_i$ for all $i\in\Z$.
	For any $n\geq 1$ we have $x_{-n}\overset\delta\sim x_n$ and there is a periodic
	$\delta$-orbit $z^n$ such that $z^n_i=x_i$ for all $|i|\leq n$.
	We have that $\shm_f(z^n)$ is a periodic point, thus $\shm_f(z^n)\in\Omega(f)$.
	Since $z^n\to x$ as $n\to\infty$, we conclude that $\shm_f(x)\in\Omega(f)$ and the proof ends.
\end{proof}

\subsection{Bracket with uniform contraction}
\label{secBracket with uniform contraction}
In \S\ref{secHyperbolic brackets} we considered brackets with hyperbolic contraction. We can prove that the bracket induced by a shift-invariant shadowing map satisfies a weaker condition.

\begin{df}
	An L-bracket has \textit{uniform contraction} if
	for all $\epsilon>0$ there is $m\geq 1$ such that
	\begin{align*}
		\dist(f^n([p,q]),f^n(q))<\epsilon,\\
		\dist(f^{-n}([p,q]),f^{-n}(p))<\epsilon
	\end{align*}
	for all $n\geq m$.
\end{df}

\begin{rmk}
	The brackets with hyperbolic contraction defined in \S\ref{secHyperbolic brackets} have uniform contraction.
\end{rmk}

\begin{prop}
	\label{propSh-invShmBrUnifCont}
	The bracket induced by a shift-invariant shadowing map has uniform contraction.
\end{prop}

\begin{proof}
As explained in \S \ref{secPropDynShift}
the shift homeomorphism has a hyperbolic bracket, \textit{i.e.}
for $x,y\in \tM$ we have
$\tdist_s(\shift^n([x,y]),\shift^n(y))=\mu^n \tdist_s(x,y)$ for all $n\geq 0$
and
$\tdist_s(\shift^{-n}([x,y]),\shift^{-n}(x))=\mu^n \tdist_s(x,y)$ for all $n\geq 0$.
Recall that this bracket is defined from $\tM\times\tM$ to $\tM$.
On $M$ we define $[p,q]_M=\shm_f(\con_f(p,q))$.
As $\shm_f$ is uniformly continuous there is $\delta>0$ such that if $\tdist_s(x,y)<\delta$ then $\dist(\shm_f(x),\shm_f(y))<\epsilon$.
Take $m\geq 1$ such that
$\mu^n\diam(\tM)<\delta$ whenever $n\geq m$.
For $n\geq m$ we have
\[
	\tdist_s(\shift^n(\orb_f(q)),\shift^n(\con_f(p,q)))<\mu^n\diam(\tM)<\delta
\]
and applying $\shm_f$ we conclude
\begin{align*}
	\epsilon&>\dist(\shm_f(\shift^n(\orb_f(q)),\shm_f(\shift^n(\con_f(p,q))))\\
	&=\dist(f^n(\shm_f(\orb_f(q))),f^n(\shm_f(\con_f(p,q))))\\
	&=\dist(f^n(q),f^n([p,q]))
\end{align*}
and the proof ends.
\end{proof}

\section{Topological hyperbolicity}
\label{Topological hyperbolicity}
This section is about shadowing maps and brackets for expansive homeomorphisms.
In \S\ref{secDinamically-invariant pseudo-orbit maps} we consider dynamically-invariant shadowing maps.
In \S\ref{secStandard expansivity} we consider shadowing maps for
topologically hyperbolic homeomorphisms.
In \S\ref{secRuelle's bracket} we study Ruelle's set of axioms for Smale spaces and topological hyperbolicity.

\subsection{Dinamically-invariant pseudo-orbit maps}
\label{secDinamically-invariant pseudo-orbit maps}
Abusing of the notation we consider $f\colon \tM\to \tM$ as $f(x)_i=f(x_i)$ for all $i\in\Z$.
By continuity, for all $\epsilon>0$ there is $\delta>0$ such that
if $x$ is a $\delta$-orbit then $f(x)$ is an $\epsilon$-orbit.
A map $\shm_f\colon \tM(f,\epsilon)\to M$ is \textit{dynamically-invariant} if $f \shm_f=\shm_f f$ on $\tM(f,\delta)$.

\begin{exa}
	\label{exaProjTildefInv}
	The projection $\shm_f(x)=x_0$ is dynamically-invariant.
	This is because
	\[
	f (\shm_f(x))=f(x_0)=\shm_f (f(x)).
	\]
	Thus, the homeomorphisms with h-shadowing considered in \S\ref{secHsh} have a dynamically-invariant shadowing map.
	This also allows us to conclude that there are dynamically-invariant pseudo-orbit maps which are not shadowing maps (any homeomorphism without the h-shadowing property).
\end{exa}

As we mentioned in \S\ref{secIntro} topologically hyperbolic homeomorphisms have a
dynamically-invariant shadowing map (which in addition is shift-invariant).

\subsection{Standard expansivity}
\label{secStandard expansivity}
Recall that $f$ is \textit{expansive} if there is $\expc>0$ such that if $p\neq q$, $p,q\in M$, then there is $i\in\Z$ such that $\dist(f^i(p),f^i(q))>\expc$.
\begin{df}
	A bracket is $f$-\textit{invariant} if
	$f([p,q])=[f(p),f(q)]$ when both sides are defined.
\end{df}

\begin{prop}
	\label{propExpDesdeCorchete}
	If $f$ admits a dynamically-invariant bracket with uniform contraction then $f$ is expansive.
\end{prop}

\begin{proof}
	Take $\expc>0$ such that $\dist(p,q)\leq\expc$ implies $\dist(p,[p,q])<\delta$ and
	$\dist(p,[q,p])<\delta$.
	We will show that $\expc$ is an expansivity constant.

	Suppose that $\dist(f^i(p),f^i(q))\leq \expc$ for all $i\in\Z$.
	Let $p_i=f^i(p)$, $q_i=f^i(q)$ and $r_i=[p_i,q_i]$ for all $i\in\Z$.
	From the $f$-invariance we know that $f^i(r_0)=r_i$.

	Given $\epsilon>0$ consider $m$ from the uniform contraction property.
	Thus we have
	\begin{align*}
		\dist(f^m(q_i),f^m(r_i))<\epsilon,\\
		\dist(f^{-m}(p_i),f^{-m}(r_i))<\epsilon
	\end{align*}
	for all $i\in\Z$.
	For $i=-m$ in the first inequality we have
	\[
	\dist(q,r_0)=\dist(q_0,r_0)=\dist(f^m(q_{-m}),f^m(r_{-m}))<\epsilon.
	\]
	As $\epsilon$ is arbitrary we conclude $q=r_0$.
	Analogously, for $i=n$ in the second inequality we conclude $p=r_0$ and $p=q$.
	Thus, $\expc$ is an expansivity constant for $f$.
\end{proof}

\begin{cor}
	\label{corshmimpexp}
	If $f$ admits a shadowing map which is dynamically-invariant and shift-invariant then $f$ is expansive.
\end{cor}

\begin{proof}
	By Proposition \ref{propSh-invShmBrUnifCont} there is a bracket with uniform contraction. Thus, the result follows from Proposition \ref{propExpDesdeCorchete}.
\end{proof}

The following result is essentially known.

\begin{thm}
	\label{thmTopHypChar}
	The following are equivalent:
	\begin{enumerate}
		\item $f$ is topologically hyperbolic,
		\item $f$ admits a shadowing map which is shift-invariant and dynamically-invariant.
	\end{enumerate}
\end{thm}

\begin{proof}
	$(1\to 2)$.
	Let $\alpha$ be an expansivity constant for $f$ and take $\delta$ such that every $\delta$ pseudo orbit
	is $\frac \alpha 2$-shadowed.
	From the expansivity we have that if $p,q$ $\frac \alpha 2$-shadow a pseudo orbit $x$ then $p=q$.
	This defines a map $\shm_f\colon \tM(f,\delta)\to M$ as $\shm_f(x)=p$ where $p$ is the unique point that $\frac \alpha 2$-shadows $x$.

	We will show that $\shm_f$ is a shadowing map for $f$ by checking the definition.
	Given $p\in M$ let $q=\shm_f(\orb_f(p))$. From expansivity we conclude that $p=q$.
	To prove that $\shm_f$ is equivariant
	take $x,y\in \tM(f,\delta)$ such that $y_i=x_{i+1}$ for all $i\in\Z$ and let
	$p=f(\shm_f(x))$ and $q=\shm_f(y)$.
	Then $\dist(f^i(f^{-1}(p)),x_i)\leq\alpha/2$ for all $i\in\Z$ and
	$\dist(f^i(q),y_i)=\dist(f^i(q),x_{i+1})\leq\alpha/2$ for all $i\in\Z$.
	This implies $\dist(f^i(p),f^i(q))\leq \alpha$ and $p=q$ from expansivity. Thus, $\shm_f$ is equivariant.

	It remains to show that $\shm_f$ is continuous.
	Suppose that $x^n\to y$ in $\tM(f,\delta)$ and
	$\shm_f(x^n)=r_n\to p$ and $\shm_f(y)=q$.
	Since $\shm_f(x^n)=r_n$ we have $\dist(f^i(r_n),x^n_i)\leq\alpha/2$ and the continuity of $f$ implies
	$\dist(f^i(p),y_i)\leq\alpha/2$ for all $i\in\Z$. Also, since
	$\shm_f(y)=q$ we have $\dist(f^i(q),y_i)\leq\alpha/2$ for all $i\in\Z$.
	Thus $\dist(f^i(p),f^i(q))\leq\alpha$, and the expansivity of $f$ implies $p=q$.
	We conclude that $\shm_f$ is continuous.
	The invariance properties were explained in \S\ref{secIntro}.

	$(2\to 1)$. The expansivity follows by Corollary \ref{corshmimpexp}. The shadowing property is a consequence of Proposition \ref{propShmSH}.
\end{proof}

\subsection{Ruelle's bracket}
\label{secRuelle's bracket}
We follow \cite{Ruelle}*{\S 7.1} concerning what D. Ruelle called \textit{Smale spaces}. He stated two sets axioms that will be recalled in detail. The first set is topological, as it does not involve any dynamics.
For $\epsilon>0$ given consider a function
$[\cdot,\cdot]\colon\Delta_\epsilon\to M$ satisfying:
\begin{enumerate}
	\item[(SS1.1)] $[\cdot,\cdot]$ is continuous and $[p,p]=p$ for all $p\in M$,
	\item[(SS1.2)] $[[p,q],r]=[p,r]$, $[p,[q,r]]=[p,r]$ when the two sides of these relations are defined.
\end{enumerate}

As explained in \cite{Ruelle}*{\S 7.1} these conditions implies a \textit{local product structure}. Indeed, if for $\delta>0$ small we define:
\begin{align*}
	V^-_p(\delta)=\{r=[p,q]\in M : \dist(p,r)\leq\delta, \dist(p,q)\leq\epsilon\},\\
	V^+_p(\delta)=\{r=[q,p]\in M : \dist(p,r)\leq\delta, \dist(p,q)\leq\epsilon\},
\end{align*}
then
\[
[\cdot,\cdot]\colon V^-_p(\delta)\times V^+_p(\delta)\to M
\]
is a homeomorphism onto a neighborhood of $p$.\footnote{Ruelle's definition of the sets $V^\pm_p(\delta)$ is slightly different as he assumes $q=[q,p]$. But assuming (SS1.2) it is easy to see that our definition is equivalent to the original one. We introduce this change in order to reduce the dependence of (SS1.2).}

Again, these axioms may not be related to any dynamics. But they can be \textit{dynamically defined} as when considering \textit{canonical coordinates}, \textit{i.e.}, $r=[p,q]$ is the point of intersection of the local stable set of $q$ and the local unstable set of $p$ if $f$ is an Anosov diffeomorphism (or any expansive homeomorphism with shadowing).
Another example of a bracket satisfying (SS1) (\textit{i.e.}, (SS1.1) and (SS1.2)) is $[p,q]=p$ for all $(p,q)\in \Delta_\epsilon$ (the projection on the first coordinate).
Also the bracket of the shift satisfies these axioms, recall \S\ref{secBracketShift}.

Then Ruelle introduces the set of axioms (SS2) with hyperbolic dynamical meaning:
\begin{enumerate}
	\item [(SS2.1)] $f$-\textit{invariance}: $f([p,q])=[f(p),f(q)]$ when both sides are defined and
	\item [(SS2.2)] \textit{hyperbolic contraction of distances}\footnote{As before, we changed this axiom. The one we give here is equivalent to the original one under the assumption of the remaining axioms.}:
	there is $\lambda>1$ such that:
	\begin{align*}
		\dist(f^n(p),f^n(q))\leq\lambda^{-n}\dist(p,q)  \text{ if }  q\in V^+_p(\delta), n\geq0,\\
		\dist(f^{-n}(p),f^{-n}(q))\leq\lambda^{-n}\dist(p,q)  \text{ if }  q\in V^-_p(\delta), n\geq 0.\\
	\end{align*}
\end{enumerate}

Ruelle define a \textit{Smale space} as a compact metric space with a bracket and a homeomorphism satisfying (SS1) and (SS2). In \cite{Ruelle}*{\S 7.3} it is explained that for a Smale space the homeomorphism is expansive and has shadowing. Thus, Ruelle's bracket gives the canonical coordinates explained before.
In the mentioned proof of expansivity it is used (SS1.2). Let us show that expansivity and (SS1.2) can be deduced from the other three axioms.

\begin{thm}
	\label{thmRuelleBExp}
	SS1.1 and SS2 implies expansivity and SS1.2.
\end{thm}

\begin{proof}
	The expansivity follows from Proposition \ref{propExpDesdeCorchete}.
For the other part we suppose that $\delta$ is small so that $4\delta$ is  an expansivity constant.
	From (SS2.2) we see that $[p,q]\in W^s_\delta(q)\cap W^u_\delta(p)$.
	Thus, $W^u_{\delta}([p,q])\subset W^u_{2\delta}(p)$ and
	\[
	[[p,q],r]\in W^s_{\delta}(r)\cap W^u_{2\delta}(p)
	\subset W^s_{2\delta}(r)\cap W^u_{2\delta}(p).
	\]
	Also $$[p,r]\in W^s_{\delta}(r)\cap W^u_{\delta}(p)\subset
	W^s_{2\delta}(r)\cap W^u_{2\delta}(p).$$
	Since $4\delta$ is an expansivity constant we conclude that $[[p,q],r]=[p,r]$.
	Analogously $[p,[q,r]]=[p,r]$ and (SS1.2) is proved.
\end{proof}

\begin{rmk}
In order to complete Ruelle's set of axioms we mention that
in \cite{Ruelle}*{\S 7.11} it is added the condition:
\begin{itemize}
	\item[(SS3)] there is $L>0$ such that
	$\dist(x,[x,y])\leq L\dist(x,y)$.
\end{itemize}
It is known that every topologically hyperbolic homeomorphism admits a metric satisfying Ruelle's axioms, including (SS3), see \cites{Fried1983,Sakai2001}.
Moreover, it is known that Dobysh's metric \cite{Dovbysh2006} and a self-similar hyperbolic metric \cite{ArSelfSimHyp} also satisfies $L\to 1$ as $\dist(x,y)\to 0$.
\end{rmk}

\section{Lee's continuous shadowing}
\label{secLeeCSh}
Following \cite{Lee03} we say that $f$ is $\mathcal{T}_0$-\textit{continuous shadowing}
if for all $\epsilon>0$ there are $\delta>0$ and a continuous map
$R_\epsilon\colon \tM(f,\delta)\to M$ such that
every $\delta$-orbit $x$ is $\epsilon$-shadowed by $R(x)$.

\begin{rmk}
 This is similar to the definition of shadowing map, however we remark that the map $R_\epsilon$ depends on $\epsilon$ and it is not requiered that $R_\epsilon(\orb_f(p))=p$ for $p\in M$.
 It is clear that if $\shm_f$ is a shadowing map then we can consider the maps $R_\epsilon$ as restrictions of $\shm_f$ to $\tM(f,\delta)$ varying $\delta$. Thus, the existence of a shadowing map implies the
 $\mathcal{T}_0$-continuous shadowing property.
 In what follow we study the
  $\mathcal{T}_0$-continuous shadowing property
  to see whether or not it implies the existence of a shadowing map.
\end{rmk}

As an orbit may not be shadowed by itself, it is natural to consider the map $G_\epsilon\colon M\to M$ defined as
$G_\epsilon(p)=R_\epsilon(\orb_f(p))$. We have some direct consequences:
\begin{enumerate}
\item $G_\epsilon$ is continuous (as it is the composition of continuous maps),
\item by definition, $G_\epsilon(p)\in \Gamma_\epsilon(p)$ for all $p\in M$ and
\item $\dist_{C^0}(G_\epsilon,Id_M)\to 0$ as $\epsilon\to 0$.
\end{enumerate}

\begin{prop}
\label{propCorrLeeCS}
 If $G_\epsilon$ is a homeomorphism for some $\epsilon$ then
 $R'_\epsilon=G_\epsilon^{-1}\circ R_\epsilon$ satisfies:
 \begin{enumerate}
 \item $R'_\epsilon$ is a pseudo-orbit map, \textit{i.e.}:
\begin{enumerate}
  \item $R'_\epsilon$ is continuous,
  \item $R'_\epsilon(\orb_f(p))=p$ for all $p\in M$,
\end{enumerate}
  \item $R'_\epsilon$ $2\epsilon$-shadows: $\dist(f^i(R'_\epsilon(x)),x_i)\leq 2\epsilon$ for all $i\in\Z$ and every $\delta$-orbit $x$.
 \end{enumerate}
\end{prop}

\begin{proof}
As a composition of continuous functions, $R'_\epsilon$ is continuous.
Also
$$R'_\epsilon(\orb_f(p))=
G_\epsilon^{-1}(R(\orb_f(p))
=
G_\epsilon^{-1}(G_\epsilon(p))=p
$$ for all $p\in M$.
Finally,
\[
 \dist(f^i(R'_\epsilon(x)),x_i)\leq
 \dist(f^i(R'_\epsilon(x)),f^i(R_\epsilon(x)))+
 \dist(f^i(R_\epsilon(x)),x_i)
\]
and $\dist(f^i(R_\epsilon(x)),x_i)\leq\epsilon$ as $R_\epsilon$ $\epsilon$-shadows by definition.
Also, since $R_\epsilon(x)=G_\epsilon(R'_\epsilon(x))$ and
$G_\epsilon(R'_\epsilon(x))\in\Gamma_\epsilon(R'_\epsilon(x))$
we have
$\dist(f^i(R'_\epsilon(x)),f^i(R_\epsilon(x)))\leq\epsilon$ for all $i\in\Z$.
Thus $\dist(f^i(R'_\epsilon(x)),x_i)\leq\epsilon+\epsilon$ for all $i\in\Z$.
\end{proof}

\begin{rmk}
	\label{rmkRepsPresOrb}
 Proposition \ref{propCorrLeeCS} means that if the maps $R_\epsilon$ induces homeomorphisms $G_\epsilon$, for all $\epsilon>0$, then we can assume $R_\epsilon(\orb_f(p))=p$ for $p\in M$ (taking $R'_\epsilon$ instead of $R_\epsilon$).
\end{rmk}

Next we show some conditions implying that the maps $G_\epsilon$ are homeomorphisms.

\begin{rmk}
 If $\epsilon$ is an expansivity constant of $f$ then $G_\epsilon$ is the indentity of $M$ (in particular it is a homeomorphism). This is because
 $G_\epsilon(p)\in\Gamma_\epsilon(p)=\{p\}$.
 But, assuming expansivity is too strong, as it is easy to see that the maps $R_\epsilon$ would have to coincide on \textit{small} pseudo orbits. That is, $R_{\epsilon_1}(x)=R_{\epsilon_2}(x)$ when both sides are defined.
\end{rmk}

For $N\geq 1$ we say that $f$ is $N$-\textit{expansive} if there is $\expc>0$ such that $\Gamma_\expc(p)$ has at most $N$ points, for all $p\in M$. This definition was introduced by C. Morales in \cite{Morales2012}.
For the proof of the next theorem we need a lemma.

\begin{lem}
\label{lemCombNsh}
 If $T\colon A\to A$ is a surjective function such that every point is eventually $N$-periodic (\textit{i.e.} for all $a\in A$ there are $0\leq i<j\leq N$ such that $T^i(a)=T^j(a)$) then
 every point is periodic and $T$ is injective.\footnote{In fact $T^{N!}=Id_A$.}
\end{lem}

\begin{proof}
First we show that any $d\in M$ is periodic with period at most $N$. As $T$ is surjective there is $c\in M$ such that $T^N(c)=d$.
Take $0\leq i<j\leq N$ such that $T^i(c)=T^j(c)$.
Applying $T^{N-i}$ we have
\[
d=T^N(c)=T^{N-i}(T^i(c))=T^{N-i}(T^j(c))=T^{j-i}(T^N(c))=
T^{j-i}(d).
\]

 To prove that $T$ is injective suppose that $T(a)=T(b)$ and take $n,m\geq 1$ such that $T^n(a)=a$ and $T^m(b)=b$.
 If $k=mn\geq 1$ then $T^k(a)=a$ and $T^k(b)=b$.
 We know that $T^{k+1}(a)=T(a)=T(b)=T^{k+1}(b)$.
 Applying $T^{k-1}$ we obtain
 $T^{k-1}(T^{k+1}(a))=T^{k-1}(T^{k+1}(b))$, thus
 $T^{2k}(a)=T^{2k}(b)$ and $a=b$. This proves that $T$ is injective.
\end{proof}

\begin{thm}
	\label{thmLeeshmN-exp}
If $f$ is $N$-expansive, has the $\mathcal{T}_0$-continuous shadowing property and $M$ is a closed manifold then
$f$ is topologically hyperbolic.
\end{thm}

\begin{proof}
 In general we know that $G_\epsilon$ is $C^0$-close to the identity, and as $M$ is a closed manifold, if $\epsilon$ is small we have that $G_\epsilon$ is surjective (see \cite{Walters1978}).

 We know that $G_\epsilon(p)\in \Gamma_\epsilon(p)$ for all $p\in M$.
 Thus, $G_\epsilon^n(p)\in \Gamma_{n\epsilon}(p)$ for all $n\geq 0$.
Suppose that $(N+1)\epsilon<\expc$.
Then the set $\{p,G_\epsilon(p),G^2_\epsilon(p),\dots,
G_\epsilon^{N}(p)\}\subset\Gamma_{\expc}(p)$ has at most $N$ points, for any $p\in M$.
Thus, for some $0\leq i<j\leq N$ we have $G_\epsilon^i(p)=
G_\epsilon^j(p)$.
This means that every point is eventually $N$-periodic for $G_\epsilon$. Applying Lemma \ref{lemCombNsh} for $T=G_\epsilon$ we conclude that $G_\epsilon$ is injective.
As it is continuous and bijective on a compact metric space its inverse is continuous and it is a homeomorphism.

This proves that the maps $G_\epsilon$ are homeomorphisms for all $\epsilon$ sufficiently small. By Remark \ref{rmkRepsPresOrb}
we can assume that the maps $R_\epsilon$ satisfies
$R_\epsilon(\orb_f(p))=p$ for all $p\in M$.

Notice that $N$-expansivity implies cw-expansivity.
If $\expc$ is a cw-expansivity constant we can suppose that $\epsilon<\expc/2$ and we have that every $\delta$-orbit $x$ is $\expc/4$-shadowed by
$R_\epsilon(x)$. This allows us to apply the techniques of the proof of Theorem \ref{thmShmCwExpImpExp} to conclude that $f$ is expansive.
\end{proof}

\begin{thm}
 If $f$ has the $\mathcal{T}_0$-continuous shadowing property and $M$ is totally disconnected then it admits a shadowing map.
\end{thm}

\begin{proof}
	Since $M$ is totally disconnected its copy in $\tM$, namely $\orb_f(M)\subset\tM$, is totally disconnected too.
	Take $\epsilon_n=1/n$ and for each $n\geq 1$ fix a map $R_{\epsilon_n}$
	defined on $\tM(f,\delta_n)$, with $\delta_n\to 0$.

	Define $U_0=\tM(f,\delta)$ and for $n\geq 1$
	take $U_n\subset U_{n-1}\cap\tM(f,\delta/(n+1))$ such that
	$U_n$ is open and closed and contains the copy of $M$.
	We will define $\shm_f\colon\tM(f,\delta)$ as follows.
	For $x\in \orb_f(M)$ we must define $\shm_f(x)=x_0$ (in order to satisfy $\shm_f(\orb_f(p))=p$ for all $p\in M$).
	For $x\in \tM(f,\delta)\setminus\orb_f(M)$
	there is a maximum $n\geq 0$ such that $x\in U_n$, and we define
	$\shm_f(x)=R_n(x)$.

	From the construction it is clear that $\shm_f$ induces shadowing and that it is continuous on $\tM(f,\delta)\setminus\orb_f(M)$.
  The continuity at $x\in\orb_f(M)$ follows from the fact that
	$\dist_{C^0}(G_\epsilon,Id_M)\to 0$ as $\epsilon\to 0$.
\end{proof}

\end{document}